\documentclass[letterpaper]{amsart}
\usepackage[utf8]{inputenc}
\pdfoutput=1

\usepackage{amsmath, amsthm, amsopn, amssymb, microtype}
\usepackage{mathrsfs} 
\usepackage[all,cmtip]{xy}
\usepackage{enumerate, tikz, etoolbox, intcalc, caption, subcaption, afterpage}
\usepackage[english]{babel}
\usepackage{enumitem}
\usepackage{csquotes}

\usepackage[colorlinks=true,urlcolor=blue,pdfborder={0 0 0}]{hyperref}
\usepackage[nameinlink]{cleveref}
\crefname{theorem}{Theorem}{Theorems}
\crefname{theorem}{Theorem}{Theorems}
\crefname{mainthm}{Theorem}{Theorems}
\crefname{lemma}{Lemma}{Lemmas}
\crefname{lem}{Lemma}{Lemmas}
\crefname{remark}{Remark}{Remarks}
\crefname{discussion}{Discussion}{Discussions}
\crefname{prop}{Proposition}{Propositions}
\crefname{defn}{Definition}{Definitions}
\crefname{corollary}{Corollary}{Corollaries}
\crefname{cor}{Corollary}{Corollaries}
\crefname{section}{Section}{Sections}
\crefname{figure}{Figure}{Figures}
\crefname{quest}{Question}{Questions}

\numberwithin{equation}{section}

\newcommand{\N}{\mathbb{N}}
\newcommand{\Z}{\mathbb{Z}}

\newcommand{\R}{\mathbb{R}}
\newcommand{\C}{\mathbb{C}}

\newcommand{\cH}{\mathcal{H}}

\newcommand{\cM}{\mathcal{M}}

\newcommand{\cR}{R}

\newtheorem{theorem}{Theorem}[section]
\newtheorem{lemma}[theorem]{Lemma}
\newtheorem{proposition}[theorem]{Proposition}
\newtheorem{corollary}[theorem]{Corollary}

\theoremstyle{definition}
\newtheorem{definition}[theorem]{Definition}

\newtheorem{remark}[theorem]{Remark}

\newcommand{\acts}{\operatorname{\curvearrowright}}

\newcommand{\abs}[1]{\left| #1 \right|}

\DeclareMathOperator{\Ad}{Ad}

\DeclareMathOperator{\fix}{Fix}

\DeclareMathOperator{\Ind}{Ind}

\newcommand{\Ch}[1]{\mathrm{Ch}\left(#1\right)}

\newcommand{\Tr}[1]{\mathrm{Tr}\left(#1\right)}
\newcommand{\PD}[1]{\mathrm{PD}_1\left(#1\right)}
\DeclareMathOperator{\Pker}{Pker}
\newcommand{\Sub}[1]{\mathrm{Sub}\left(#1\right)}

\title{Non-uniform higher-rank lattices are character rigid}
\author[A. Dogon, M. Glasner, Y. Gorfine, L. Hanany, A. Levit]
{\mbox{Alon Dogon, Michael Glasner, Yuval Gorfine, Liam Hanany and Arie Levit}}

\begin{document}
\maketitle
\begin{abstract}
We establish character rigidity for all non-uniform higher-rank irreducible lattices in semisimple groups of characteristic other than $2$. This implies stabilizer rigidity for probability measure preserving actions and rigidity of invariant random subgroups, confirming a conjecture of Stuck and Zimmer for non-uniform lattices in full generality.
\end{abstract}

\section{Introduction}
Irreducible lattices in higher rank semisimple Lie groups are rigid in several ways. One celebrated such result is the normal subgroup theorem of Margulis \cite{margulis1978quotient,margulis1979finiteness}, which says that residual finiteness is the only obstacle preventing those lattices from being simple. More precisely, every non-trivial normal subgroup of an irreducible lattice in a higher-rank center-free semisimple Lie group has finite index. Notably, the normal subgroup theorem has been established by Margulis for all such irreducible lattices, including ones without Kazhdan's property (T). 

A strengthening of the normal subgroup theorem is the theorem of Stuck and Zimmer \cite{stuck1994stabilizers}. It states that every ergodic probability measure preserving action of an irreducible lattice in a higher-rank center-free semisimple Lie group with property (T) is either essentially transitive or essentially free. While stronger than the normal subgroup theorem, the Stuck--Zimmer theorem is not known in general for irreducible lattices in products of rank-one groups. This remains an important open question  \cite[Conjecture 5.3]{gelander2024things}. For some tentative recent progress in that direction see  \cite{bader2024spectral}.

A trace on a discrete group is a normalized, conjugation invariant, positive-definite function. A character is an indecomposable trace. A standard construction, due to Vershik \cite{vershik2011nonfree}, associates a trace on a discrete group to a probability measure preserving action. Via this construction, the normal subgroup and Stuck--Zimmer theorems have yet another far reaching generalization called  character rigidity. It was first conjectured by Connes \cite{Jo00}. 

\begin{definition}
\label{def:character rigidity}
    A discrete group $\Gamma$ is  \textit{character rigid} if every character
 of $\Gamma$ is either finite-dimensional (i.e. is of the form $\frac{1}{\dim \pi} \mathrm{tr} \circ \pi$ for some finite-dimensional irreducible unitary representation $\pi$), or vanishes outside the center $\mathrm{Z}(\Gamma)$.
\end{definition}

\label{def page: semisimple group}
For us, a \textit{semisimple group} is a group of the form $H=\prod_{i=1}^n\mathbf{H}_i(k_i)$, where each $k_i$ is a local field, and $\mathbf{H}_i$ is a semisimple $k_i$-algebraic group. The rank of this group is $\mathrm{rank}(H) = \sum_{i=1}^n \mathrm{rank}_{k_i}(\mathbf{H}_i)$.  Our main result is the following.

\begin{theorem}\label{main_thm}
Let $H$ be a semisimple group with $\mathrm{rank}(H) \ge 2$. Assume that $H$ admits an almost simple factor $L$ with $\mathrm{rank}(L) = 1$ and that $\mathrm{char}(k_1) \neq 2$.
Then every irreducible non-uniform lattice in $H$ is character rigid.
\end{theorem}

Character classification results were initially proved for the infinite symmetric group $S_\infty$ by Thoma \cite{thoma1964unzerlegbaren} and Vershik--Kerov \cite{vershik1981characters}, the groups $\mathrm{GL}_\infty(k)$ where $k$ is a finite field by Skudlarek \cite{skudlarek1976unzerlegbaren} and the groups $\mathrm{SL}_n(k)$ where $k$ is an infinite field and $n \ge 3$ by Kirillov \cite{kirillov1965positive}. The first character rigidity result in the context of higher-rank lattices was obtained by Bekka \cite{BekkaCharRig} for the arithmetic group $\mathrm{SL}_n(\mathbb{Z)}$ where $n \ge 3$, and generalized by Peterson \cite{PetersonCharRig} and later by Bader--Boutonnet--Houdayer \cite{BBH} to higher-rank lattices with property (T). See also \cite{lavi2023characters} for higher-rank universal lattices.

In line with the Stuck--Zimmer conjecture, the phenomenon of character rigidity is conjectured to hold for every irreducible higher-rank lattice \cite[Problem 4.2]{houdayer2021noncommutative}. Generally speaking, up to the present work character rigidity was known provided the enveloping semisimple group admits a factor with property (T) \cite{BBH}. An exceptional result that does not fall into this description is the  work of Peterson and Thom \cite{PetersonThom}. They show that the group $\mathrm{SL}_2(R)$ is character rigid, where $R$ is a localization of the ring of integers of a number field admitting  an infinite order unit. Taking into account Margulis' arithmeticity theorem and Tits' classification of semisimple algebraic groups over number fields, it follows from \cite{PetersonThom} that any non-uniform irreducible lattice in a product of $\mathrm{SL}_2(k_i)$'s is character rigid, where the $k_i$'s are local fields of characteristic zero. We refer the reader  to Appendix \ref{appen} below for a further discussion of Tits' classification.

Our result is inspired by the work of Peterson and Thom \cite{PetersonThom}. We extend it to arbitrary non-uniform lattices in products of rank-one groups, in the generality outlined in Theorem \ref{main_thm}. More precisely, we show that if:
\begin{enumerate}
    \item $F$ is a global field of characteristic different from $2$;
    \item $\mathbf{G}$ is a connected, simply connected algebraic $F$-group with $\mathrm{rank}_F(\mathbf{G}) = 1$;
    \item $S$ is a finite set of places of $F$ that contains all  Archimedean places, and $\abs{S} \geq 2$;
    \item $R=F(S)$ is the ring of $S$-integral elements in $F$;
\end{enumerate}
then the arithmetic group $\Gamma = \mathbf{G}(R)$  is character rigid. The arithmeticity theorem shows that all non-uniform irreducible lattices  
in a semisimple group having a factor of rank-one are commensurable to an arithmetic group of this form (see Proposition \ref{prop:notationIsGood}). We treat character rigidity as an abstract property, and show that it is invariant under commensurability and central extensions for ICC groups having countably many finite-dimensional unitary representations. So our result holds regardless of the chosen representative from a commensurability class of an arithmetic group.
Combined with the  above-mentioned existing results by other authors, we get:

\begin{theorem}\label{main_thm_general}
Let $H$ be a semisimple group with $\mathrm{rank}(H)  \ge 2$.
Let $\Gamma$ be an irreducible lattice in $H$. Assume that (at least) one of the following holds:
    \begin{enumerate}
        \item the lattice $\Gamma$ is non-uniform and the characteristic of the local fields involved in $H$ is different from $2$, or
        \item the group $H$ has a non-compact factor with Kazhdan's property (T).
    \end{enumerate}
    Then $\Gamma$ is character rigid.
\end{theorem}

We partially resolve the long-standing open problem of Stuck and Zimmer for higher-rank irreducible lattices, fully establishing it in the non-uniform case.

\begin{corollary}
\label{cor:SZ intro}
Let $\Gamma$ be an non-uniform irreducible lattice as in Theorem \ref{main_thm_general}.
 Every ergodic probability measure preserving action $\Gamma \acts (X, \mu)$ is either essentially transitive or has $\mu(\fix(g)) = 0$ for every element $g \in \Gamma \setminus \mathrm{Z}(\Gamma)$. 
\end{corollary}

An \emph{invariant random subgroup} of $\Gamma$ is a Borel probability measure on the Chabauty space $\Sub{\Gamma}$ of subgroups invariant under the $\Gamma$-action via conjugation \cite{abert2014kesten}.
It follows from Corollary  \ref{cor:SZ intro} that every ergodic invariant random subgroup of the non-uniform irreducible higher-rank lattice $\Gamma$ is either the uniform measure supported on the conjugacy class of some finite-index subgroup or the Dirac measure on some central subgroup. 
More so, Theorem \ref{main_thm_general} implies that the  lattices we consider are \emph{charfinite}, a notion introduced by Bader--Boutonnet--Houdayer--Peterson \cite{BBHP} (see Remark \ref{rem:charfinite} below).
It follows that the lattice $\Gamma$ enjoys several remarkable rigidity properties pertaining to uniformly recurrent subgroups and unitary representations, as outlined in \cite[\S3]{BBHP}.

The observation that character rigidity is invariant under central extensions (as long as the group has countably many finite-dimensional representations) does not require the extension to be finite.
Using the fact that irreducible lattices in higher-rank semisimple Lie groups (with no finite center  assumptions) have at most countably many finite-dimensional unitary representations, we get the following result.

\begin{theorem}\label{inf_ctr}
Let $H$ be a connected semisimple Lie group with $\mathrm{rank}_\R(H) \ge 2$,  no compact factors and arbitrary center. Let  $\Gamma$ an irreducible lattice in $H$. Assume either that  $\Gamma$ is non-uniform or that $H$ has a factor with Kazhdan's property (T). Then  the lattice $\Gamma$ is character rigid, and satisfies the conclusion of Corollary \ref{cor:SZ intro}.
\end{theorem}

Our interest in lattices with infinite center stems from the fact they are not necessarily residually finite, and as such, are possible candidates in the search for non-sofic or non-hyperlinear groups, see \cite{dogon2023flexible,chapman2024conditional,gohla2024high,dogon2025connections}.
These lattices are also the first examples of groups with infinite center which are operator-algebraic superrigid, in the sense of \cite{BekkaCharRig}.
Note that $W^*$-superrigid groups with infinite center have only recently been constructed  
\cite{donvil2024w, zbMATH07991641}.

\subsection{Main ideas}

A main ingredient to our proof is Theorem \ref{theorem:on characters of G admitting AU and AV}. It deals with arithmetic lattices $\Gamma = \mathbf G(R)$ as considered above, so that in particular $\mathrm{rank}_F(\mathbf G) = 1$ and $|S| \ge 2$.  It states that any weakly mixing trace of $\Gamma$ restricts to a mixing trace on any $F$-split torus. Moreover, its decay is uniform over  elements of the form $a^n u$, where $a$ generates a Zariski-dense subgroup of the $F$-split torus and $u$ is unipotent. Dirichlet's unit theorem guarantees that such an element $a$ exists. This is where our assumption that the lattice is non-uniform crucially comes into play.

Theorem \ref{theorem:on characters of G admitting AU and AV} is a generalization of the \enquote{main proposition} from the work of Peterson and Thom for $\mathrm{SL}_2$, i.e. \cite[Proposition 2.2]{PetersonThom}. In a nutshell, we rely on a careful analysis of the character theory of certain solvable arithmetic groups.
We define and study congruence characters of unipotent arithmetic groups. For an arithmetic unipotent group $U$ acted upon by a  semisimple element $a$, we show that a  character is congruence if and only if its $a$-orbit in the Thoma dual is finite (Proposition \ref{prop:rational if and only of finite orbit}). This generalizes the fact that rational points on a torus are precisely the periodic points for an ergodic torus automorphism. Moreover, we prove that characters of the semidirect product $M = \left<a\right>\ltimes U$ which restrict to a non-congruence trace on the unipotent subgroup $U$ are induced from it  (Proposition \ref{prop:characters of AU}). Now Theorem \ref{theorem:on characters of G admitting AU and AV} is inferred from an interplay of the mixing behavior of the split torus and the properties of the restriction of the trace to opposite parabolic subgroups.

Interestingly, we are able to provide two  different proofs of character rigidity (Theorem \ref{main_thm}), both relying on Theorem \ref{theorem:on characters of G admitting AU and AV}. The first proof is self-contained. It uses the structure of the Bruhat decomposition in rank one, and a vanishing lemma of Bekka (see Lemma \ref{lemma:generalized bekka lemma} below). The second proof is softer but is not self-contained. It uses the notion of \emph{charmenability} in order to reduce the question from a general non-uniform higher-rank lattice to a group of the form $\mathrm{SL}_2(R)$, where the work of Peterson--Thom applies directly.

\subsection{Structure of the paper}
Preliminaries on character theory are discussed in \S\ref{sec:2}, including disjointness of traces (\S\ref{sub:disjoint}) and weakly mixing and mixing traces (\S\ref{subsec:2.2}). In \S\ref{sec:3} we recall the classical construction of arithmetic lattices, and explain which algebraic groups over global fields are relevant for the lattices considered in this paper. We also fix some notations to be used throughout (see \S\ref{standing assumption}). In \S\ref{sec:4} we study the property of character rigidity under commensurability and central extensions. We infer that character rigidity is a commensurability invariant for higher-rank lattices. In \S\ref{sec:5} we start developing the theory needed to prove our main Theorem \ref{theorem:on characters of G admitting AU and AV}. We study characters and traces of certain subgroups of solvable arithmetic groups, and introduce the notion of congruence traces. In \S\ref{sec:6} we formulate and prove Theorem \ref{theorem:on characters of G admitting AU and AV} regarding the decay of weakly mixing characters on split tori, relying on machinery from \S\ref{sec:5}. In \S\ref{sec:7} we establish character rigidity and provide two proofs of Theorem \ref{main_thm}. The final  \S\ref{sec:8} is dedicated to proving the remaining results (Theorem \ref{main_thm_general} and  Corollaries \ref{cor:SZ intro} and \ref{inf_ctr}) as well as some further implications. In Appendix \ref{appen} we recall Tits' classification of semisimple algebraic groups, and give explicit descriptions of the relevant lattices in characteristic zero. 

\subsection{Acknowledgments} The authors would like to thank Uri Bader for suggesting to us that the ideas of Peterson and Thom might be applicable to a more general class of lattices. We would also like to thank Tomer Konforty for numerous  discussions during an earlier stage of this work,  Itamar Vigdorovich for conversations regarding disjointness and characters of solvable groups, Tsachik Gelander and Yair Glasner for  comments and suggestions. Lastly, the authors would like to thank Tel-Aviv University, where a major part of the research was done, for its hospitality.

\subsection{Funding information}
A.D. was supported by a Clore Scholars grant and ERC grant no. 882751.
L.H. was supported by ERC grant no. 882751.
A.L. acknowledges the support of BSF grant $\#2022105$ and ISF grant $\#1788/22$.




\section{Character theory}\label{sec:2}
We recall some basic notions from the character theory of countable groups. We relate it to the unitary representation theory of the group, and to the representation theory of the group into tracial von Neumann algebras. The notions of disjointness and mixing are also considered. We refer the reader to \cite[Chapter 11]{BdlH} for a more thorough exposition to character theory, and to \cite{Popa} for more on tracial von Neumann algebras.
Throughout the section $\Lambda$ denotes a countable group.

\subsection{Basic notions}\label{subsec:2.1}

A function $\varphi : \Lambda \to \C$ is called \emph{positive definite} if for every $n \in \N$ and every choice of elements $g_1,\ldots,g_n \in \Lambda$ the matrix $\varphi(g_i^{-1}g_j)_{i,j}$ is positive definite. The set of positive definite functions on $\Lambda$ is denoted by $\mathrm{PD}({\Lambda})$. Any such function $\varphi \in \mathrm{PD}({\Lambda})$ satisfies $\|\varphi\|_\infty = \varphi(e_\Lambda)$.  A function $\varphi \in \mathrm{PD}({\Lambda})$ is called \emph{normalized} if $\|\varphi\|_\infty = \varphi(e_\Lambda)= 1$. The set of normalized positive definite functions on $\Lambda$ is denoted by $\PD{\Lambda}$.
\begin{definition}
\label{def:character}
A \emph{trace} $\varphi$ on the group $\Lambda$  is a conjugation-invariant function $\varphi \in \PD{\Lambda}$. The set of traces on $\Lambda$ is denoted by $\Tr{\Lambda}$.
\end{definition}

Alternatively, a trace $\varphi \in \ell^\infty(\Lambda)$ on the group $\Lambda$ can be viewed as a positive definite linear functional on the Banach $*$-algebra $\ell^1(\Lambda)$ whose norm as a functional is $1$ and which satisfies
$\varphi(gh)=\varphi(hg)$ for every pair of elements $g,h \in \Lambda$. 

The set $\Tr{\Lambda}$ is a compact convex set with respect to the topology of pointwise convergence.
We denote by $\Ch{\Lambda}$ the set of extremal points of $\Tr{\Lambda}$. Elements of $\Ch{\Lambda}$ are called \emph{characters}\footnote{Some authors use the term character for a (non-extremal). We reserve the word character only for extremal traces.} of $\Lambda$. The set $\Ch{\Lambda}$ is called the \textit{Thoma dual} of $\Lambda$.

The space $\Tr{\Lambda}$ is a metrizable Choquet simplex. It follows that for every trace $\varphi \in \Tr{\Lambda}$ there is a unique probability measure $\mu_\varphi \in \mathrm{Prob}(\Ch{\Lambda})$ such that:
\[ \varphi = \int_{\Ch{\Lambda}}  \psi \; \mathrm{d}\mu_\varphi(\psi)\]
The measure $\mu_\varphi$ is called the \emph{Bochner transform} of $\varphi$.

Associated to a positive definite function $\varphi \in \PD{\Lambda}$ is a canonical cyclic unitary representation called the \emph{Gelfand--Naimark--Segal (GNS)} construction.
Specifically, the \emph{GNS-triple} associated to the positive-definite function $\varphi$ is a triple $(\cH_\varphi, \pi_\varphi,\xi_\varphi)$ consisting of a Hilbert space $\cH_\varphi$, a unitary representation $\pi_\varphi : \Lambda \to U(\cH_\varphi)$ and a cyclic unit vector $\xi_\varphi \in \cH_\varphi$ satisfying 
$$ \varphi(g) = \langle\pi_\varphi(g) \xi_\varphi,\xi_\varphi\rangle \quad \forall g \in \Lambda.$$ 
The \emph{von Neumann algebra generated by $\pi_\varphi$} is  $\mathcal M_\varphi = \pi_\varphi(\Lambda)'' \leq B(\cH_\varphi)$.

Assume that the positive-definite function $\varphi$ is a trace, namely $\varphi \in \Tr{\Lambda}$. Then there exists a canonical unitary representation $\rho_\varphi:\Lambda \to U(\cH_\varphi) $ commuting with $\pi_\varphi$, such that $(\cH_\varphi, \rho_\varphi,\xi_\varphi)$ is another GNS-triple associated to $\varphi$.
Further, the linear functional on $B(\cH_\varphi)$ given by  $T \mapsto \langle T\xi_\varphi, \xi_\varphi\rangle$ yields a faithful, normal tracial state $\tau$ on $\cM_\varphi$,  whose composition with $\pi_\varphi$ recovers $\varphi$.
Such a state on a von Neumann algebra is called a \emph{trace}.
With slight abuse of notation, we will denote the resulting tracial von Neumann algebra by $(\cM_\varphi, \tau)$.
A fundamental result of Thoma relates characters of the group $\Lambda$ to its representation theory into tracial von Neumann algebras.
\begin{theorem}[Thoma \cite{thoma1964unitare,thoma1964positiv}]\label{thm:Thoma}
\label{thm:thoma}
Let $\Lambda$ be a countable group and   $\varphi:\Lambda \to \C$ be a function. The following are equivalent:
    \begin{enumerate}
        \item $\varphi \in \Tr{\Lambda}$; namely $\varphi$ is a trace on the group $\Lambda$.
        \item There exists a pair of commuting unitary representations $\pi, \rho: \Lambda \to U(\cH)$ on some Hilbert space $\mathcal{H}$ and a unit vector $\xi \in \cH$ such that $\pi(g)\rho(g) \xi = \xi$ and $\varphi(g) = \langle \pi(g) \xi, \xi\rangle =  \langle \rho(g) \xi, \xi\rangle$ for all $g \in \Lambda$.
        \item There exists a tracial von Neumann algebra $(\cM,\tau)$ and a unitary representation $\pi: \Lambda \to U(\cM)$ such that $\pi(\Lambda)'' =\cM$ and  $\varphi(g) = \tau(\pi(g))$ for all $g \in \Lambda$.
    \end{enumerate}
Given a trace $\varphi \in \Tr{\Lambda}$, the unitary representation $\pi: \Lambda \to U(\cM)$ inducing $\varphi$ is unique up to an equivariant, trace preserving $*$-isomorphism of $(\cM,\tau)$.
The trace $\varphi$ is a character if and only if $\cM$ is a factor.
\end{theorem}

\subsection{Disjoint traces}
\label{sub:disjoint}

For a pair of traces $\varphi,\psi \in \Tr{\Lambda}$ we say\footnote{Our notation for dominated  traces is slightly non-standard. Indeed, some authors prefer to work with non-normalized traces instead.} that \emph{$\varphi$ dominates $\psi$} and write  $\psi \le \varphi$ if there is a scalar $t \in \left(0,1\right]$ such that $\varphi - t\psi \in \mathrm{PD}(\Lambda)$ \cite[\S11.C]{BdlH}.

\begin{definition}
\label{def:disjoint traces}
A pair of traces $\varphi_1, \varphi_2\in\Tr{\Lambda}$ are called \emph{disjoint} if there does not exist a trace $\psi \in \Tr{\Lambda}$ with $\psi \le \varphi_1$ and $\psi \le \varphi_2$.
\end{definition}

Recall that two unitary representations are called \emph{disjoint} if they do not admit isomorphic subrepresentations. The next proposition relates the notions of disjointness for traces, unitary representations and measures. We were not able to locate an explicit convenient reference in the literature, and include a proof for completeness. 

\begin{proposition}\label{prop:disj traces and gns}
Let $\varphi_1,\varphi_2 \in \Tr{\Lambda}$ be a pair of traces. The following three conditions are equivalent:
\begin{enumerate}
    \item The two traces $\varphi_1$ and $\varphi_2$ are disjoint.
    \item The two GNS-representations $\pi_{\varphi_1}$ and $\pi_{\varphi_2}$ are disjoint.
    \item The two Bochner transforms $\mu_{\varphi_1}$ and $\mu_{\varphi_2}$ are mutually singular as measures on the Thoma dual $\Ch{\Lambda}$.
\end{enumerate}
Furthermore, if a trace $\varphi \in \Tr{G}$ can be written as a non-trivial convex combination $\varphi = t\varphi_1 + (1-t)\varphi_2$ where $\varphi_1,\varphi_2\in\Tr{\Lambda}$ are disjoint and $t \in \left(0,1\right)$, then we have a disjoint decomposition of the corresponding GNS representation $\pi_\varphi = \pi_{\varphi_1} \oplus \pi_{\varphi_2}$. The cyclic vector $\xi_\varphi$ decomposes as a sum of cyclic vectors $\xi_\varphi = \xi_{\varphi_1} + \xi_{\varphi_2}$ such that
$$t\varphi_1(\cdot) = \left<\pi_{\varphi_1}(\cdot)\xi_{\varphi_1},\xi_{\varphi_1}\right> \quad \text{and} \quad (1-t)\varphi_2(\cdot) = \left<\pi_{\varphi_2}(\cdot)\xi_{\varphi_2},\xi_{\varphi_2}\right>.$$
\end{proposition}

While the statement is intuitive, its proof is technical, and we encourage the reader to skip it on first reading.

\begin{proof}[Proof of Proposition \ref{prop:disj traces and gns}]
We establish the equivalence of the three stated conditions by first showing $(1) \Leftrightarrow (2)$ and then $(1) \Leftrightarrow (3)$.

Let us show $(1) \Leftrightarrow (2)$. In one direction, assume that the two traces $\varphi_1$ and $\varphi_2$ are disjoint. Denote by $(\cM_{\varphi_i}, \varphi_i)$ the associated tracial von Neumann algebras for $i \in \{1,2\}$. If, arguing towards contradiction, the two GNS representations $\pi_{\varphi_1}$ and $\pi_{\varphi_2}$ were not disjoint, they would have a common subrepresentation $\rho$.
 Denote $\mathcal{M}_\rho = \rho(\Lambda)''$, so that $\mathcal{M}_\rho$ is a tracial von Neumann algebra. Fix some trace $\tau$ on $\mathcal{M}_\rho$.
There exist central projections $p_i \in \mathrm{Z}(\mathcal{M}_{\varphi_i})$ and $*$-isomorphisms $\theta_i :\mathcal{M}_{\varphi_i} p_i \to \mathcal M_{\rho}$ such that $\theta_i(\pi_{\varphi_i}(g)p_i) = \rho(g)$ for all $g \in \Lambda$ and  $i\in\{1,2\}$  \cite[Proposition 6.B.5]{BdlH}.
The functions $\psi_i(\cdot) = \varphi_i(\pi_{\varphi_i}(\cdot) p_i)$ are traces on $\Lambda$ satisfying $\psi_i \le \varphi_i$ for $ i\in \{1,2\}$ \cite[Lemma~11.C.2]{BdlH}. 
According to Theorem \ref{thm:thoma} there are traces $\tau_1$ and $\tau_2$ on the von Neumann algebra $\mathcal M_\rho $ such that  $\psi_i= \tau_i \circ \rho$ for $i \in \{1,2\}$.  We will arrive at a contradiction by showing that 
$\psi_1$ and $\psi_2$ are not disjoint.

Let $\mathcal A_\rho $ the center of $\mathcal M_\rho $. Consider the standard representation $L^2(\mathcal M_\rho , \tau)$.
There is a one-to-one correspondence between normal tracial states on $\mathcal M_\rho $ and positive elements of norm one in the subspace $L^2(\mathcal A_\rho , \tau) \subset L^2(\mathcal M_\rho , \tau)$ \cite[Proposition~7.3.9]{Popa}. 
As the von Neumann algebra $\mathcal A_\rho $ is abelian, there is a standard probability space $(X, \mu)$ such that $(\mathcal A_\rho , \tau)$ is tracially isomorphic to $(L^\infty(X, \mu), \mu)$. The standard representation $L^2(\mathcal A_\rho , \tau)$ is given by the action of $L^\infty(X, \mu)$ on $L^2(X, \mu)$  \cite[Corollary 3.2.2]{Popa}. Hence $\tau_1$ and $\tau_2$ correspond to  positive norm one functions $f_1,f_2 \in L^2(X,\mu)$. 
The functions $f_1$ and $f_2$ are $\mu$-almost surely positive  by the faithfulness of $\tau_1$ and $\tau_2$.
The $\mu$-almost surely positive function $f=\min \{f_1,f_2\}$ corresponds to a non-zero positive, tracial normal functional $\tau'$ on $\mathcal M_\rho $. The normalized pullback $\psi=\frac{\tau' \circ \rho}{\tau'(1)} $ is a trace satisfying $\psi \leq \psi_i$ for $i \in \{1,2\}$. We arrive at a contradiction.

For the opposite direction, if the two GNS representations $\pi_{\varphi_1}$ and $\pi_{\varphi_2}$ are disjoint, then the two traces $\varphi_1$ and $\varphi_2$ must also be disjoint by 
\cite[Lemmas 1.B.11 and 11.C.2]{BdlH}.

The equivalence $(1) \Leftrightarrow (3)$ follows from the fact the Bochner transform sets up an affine isomorphism of convex sets between $\Tr{\Lambda}$ and $\mathrm{Prob}(\Ch{\Lambda})$, via Choquet theory. Under the Bochner transform, disjoint traces $\varphi_1$ and $\varphi_2$ correspond to measures $\mu_{\varphi_1}$ and $\mu_{\varphi_2}$ with the following property: there exists no non-zero positive measure $\nu$ with $\nu(A) \leq \mu_{\varphi_i}(A)$ 
for all Borel sets $A$ and $i \in \{1,2\}$. This property characterizes mutually singular measures.  

Finally, consider a non-trivial convex combination $\varphi = t\varphi_1 + (1-t)\varphi_2$ where $\varphi,\varphi_1,\varphi_2\in\Tr{\Lambda}$, $t \in \left(0,1\right)$ and the two traces $\varphi_1$ and $\varphi_2$ are disjoint. Let $(\mathcal H_{\varphi},\pi_\varphi,\xi_\varphi)$ be the GNS-triple associated to $\varphi$. According to \cite[Lemma~11.C.2]{BdlH} there is an operator $T$ in the center of the von Neumann algebra $\mathcal M_\varphi$ with $0 \le T \le 1$ such that the two vectors $\xi_{\varphi_1} = T\xi_\varphi$ and $\xi_{\varphi_2} = (1-T)\xi_{\varphi}$ are cyclic vectors for  GNS representations of $\varphi_1$ and $\varphi_2$. As the two traces $\varphi_1$ and $\varphi_2$ are disjoint, their GNS representations are disjoint by the above. Denote the corresponding Hilbert subspaces of $\mathcal{H}_\varphi$ by $\mathcal H_{\varphi_1}$ and $\mathcal H_{\varphi_2}$. We claim that $\mathcal H_{\varphi_1} \perp \mathcal H_{\varphi_2}$. Indeed, since $T$ is central, the orthogonal projection $p_i$ onto $\mathcal H_{\varphi_i}$ is central. Therefore each projection $p_i$ intertwines the representations $\mathcal{H}_{\varphi_1}$ and $\mathcal{H}_{\varphi_2}$. But $\mathcal{H}_{\varphi_1}$ and $\mathcal{H}_{\varphi_2}$ are disjoint so that $\mathcal{H}_{\varphi_1} \subseteq \ker{p_2}$ and $\mathcal{H}_{\varphi_2} \subseteq \ker{p_1}$, establishing the claim. Since the cyclic vector $\xi_\varphi$ satisfies $\xi_\varphi \in \mathcal{H}_{\varphi_1} + \mathcal{H}_{\varphi_2}$ we get $\mathcal{H}_\varphi = \mathcal{H}_{\varphi_1} \oplus \mathcal{H}_{\varphi_2}$ as desired.
\end{proof}

\subsection{Mixing and weakly mixing traces}\label{subsec:2.2}

Let $\varphi \in \Tr{\Lambda}$ be a trace. It is called \emph{finite-dimensional} if its GNS representation is finite-dimensional, and \emph{infinite-dimensional} otherwise. The trace $\varphi$ is called  \emph{von Neumann amenable} if the von Neumann algebra $\cM_\varphi$ is amenable; see \cite[Chapter 10]{Popa}. In  case  $\varphi$ is a character, its amenability is equivalent to $\cM_\varphi$ being isomorphic to the hyperfinite $\text{II}_1$-factor $\mathcal R$ \cite[Section 1.6]{Popa}.

The trace $\varphi \in \Tr{\Lambda}$ is called \emph{weakly mixing} if it satisfies one the following equivalent conditions:
\begin{itemize}
    \item the GNS representation $\pi_\varphi$ is  weakly mixing, in the sense that $\pi_\varphi$ admits no finite-dimensional subrepresentations,
    \item there is no finite-dimensional trace $\psi \in \Tr{\Lambda}$ with $\psi \le \varphi$, or
    \item for every $\varepsilon > 0$ and every finite subset $A \subset \Lambda$  there exists an element $g \in \Lambda$ such that $|\varphi(gh)| \leq \varepsilon$ for every $h \in A$.
\end{itemize}
Given a trace $\varphi \in \Tr{\Lambda}$, we can uniquely decompose its Bochner transform $\mu_{\varphi}$ as a convex combination of an atomic measure supported on finite-dimensional characters and another mutually singular measure. By  Proposition \ref{prop:disj traces and gns} this will correspond to a convex  decomposition $\varphi=t\varphi_{\mathrm{wm}}+(1-t)\varphi_{\mathrm{nwm}}$ of $\varphi$ into disjoint weakly mixing and non-weakly mixing traces. The trace $\varphi_{\mathrm{nwm}}$ is a convex combination of countably many finite-dimensional characters. Similarly,   a unitary representation can be decomposed as a disjoint sum $\pi=\pi_{\mathrm{wm}} \oplus \pi_\mathrm{nwm}$ of its weakly mixing and non-weakly mixing parts, where $\pi_\mathrm{nwm}$ is a direct sum of finite-dimensional unitary representations. The decomposition of the trace yields the corresponding decomposition of its GNS representation, in alignment with Proposition \ref{prop:disj traces and gns}. 

\begin{remark}
Generally speaking, the Bochner transform of a weakly mixing trace may be supported on finite-dimensional characters, but it cannot have any atoms at such characters.
\end{remark}

The trace $\varphi \in \Tr{\Lambda}$ is called \emph{mixing} if $\lim_n \varphi(g_n) = 0 $ for every sequence of pairwise distinct elements $g_n \in \Lambda$. The trace $\varphi$ is mixing if and only if its GNS representation is mixing, in the sense that all of its matrix coefficients decay to zero at infinity on $\Lambda$.

\begin{lemma}
\label{lemma:unitary decomposes into mixing and non-mixing}
Let $\pi$ be a unitary representation of the group $\Lambda$. Then $\pi$ uniquely decomposes as a direct sum $\pi = \pi_\mathrm{mix} \oplus \pi_\mathrm{nm}$, so that $\pi_\mathrm{mix}$ is mixing and no subrepresentation of $\pi_\mathrm{nm}$ is mixing.
\end{lemma}
\begin{proof}
Let $\mathcal{H}$ be the underlying Hilbert space of the representation $\pi$. Denote 
$$ \mathcal{H}_\mathrm{mix} = \{v \in \mathcal{H} \: :\: \lim_n \left<\pi(g)v,v\right> = 0 \quad \forall \text{$g_n\in \Lambda$ pairwise distinct} \}. $$
We leave to the reader the straightforward verification that $\mathcal{H}_\mathrm{mix}$ is a $\pi(\Lambda)$-invariant Hilbert subspace of $\mathcal{H}$. Take $\pi_{\mathrm{mix}}$ to be the subrepresentation  corresponding to this subspace, and $\pi_{\mathrm{nm}}$  the subrepresentation corresponding to the orthogonal complement $\mathcal{H}_\mathrm{mix}^\perp$.
\end{proof}

\begin{remark}
Building upon Lemma \ref{lemma:unitary decomposes into mixing and non-mixing}, in can be further shown that any trace $\varphi \in \Tr{\Lambda}$ has a unique decomposition as a convex combination of its mixing and non-mixing parts. We will not need this fact.
\end{remark}

\subsection{Operations with traces}\label{subsec:2.3}
Beyond taking convex combinations, one can also take tensor products and complex conjugates of traces.
If $\varphi \in \Tr{\Lambda}$ then the \emph{conjugate trace} is given by $\overline{\varphi}(g) = \overline{\varphi(g)} = \varphi(g^{-1})$ for all $g \in G$. 
If $\varphi_1, \varphi_2 \in \Tr{\Lambda}$ then the \emph{tensor product trace} is given by $\varphi_1\otimes \varphi_2(g) = \varphi_1(g) \cdot \varphi_2(g)$ for all $g \in G$. The conjugate and tensor product traces are indeed traces on $\Lambda$.
It is straightforward to see that the contragredient representation $\overline{\pi}_\varphi$ induces the conjugate trace, and that the tensor product representation  $\pi_{\varphi_1} \otimes \pi_{\varphi_2} $ induces the tensor product trace.  In particular, if $\varphi$ is a trace, then so is $| \varphi|^2 = \varphi \otimes \overline \varphi$.

\begin{definition}\label{def:ker_of_trace}
Let $\varphi$ be a trace on the group $\Lambda$.
 The \emph{kernel} of $\varphi$ is
        \[
        \ker\varphi=\left\{ g\in \Lambda \: : \: \varphi(g)=1\right\}.
        \]
 The \emph{projective kernel} of $\varphi$ is 
        \[
        \Pker\varphi=\ker |\varphi|^2 = \left\{ g\in \Lambda \: : \: |\varphi(g)|=1\right\}.
        \]
The trace $\varphi$ is called \emph{faithful} if $\ker\varphi=\{e\}$. It  is called \emph{trivial} if $\ker \varphi = \Lambda$, so that $\varphi \equiv 1$.
\end{definition}

The kernel and projective kernel of any trace $\varphi \in \Tr \Lambda$ are normal subgroups of $\Lambda$. The traces  $\varphi$ and $|\varphi|^2$ respectively descend to faithful traces on the quotient groups $\Lambda/\ker\varphi$ and $\Lambda / \Pker \varphi$. For all this see \cite[Lemma 12.A.1]{BdlH}.

\begin{lemma}[Schur's lemma for characters]
\label{lemma:schur}
Let $\varphi \in \Ch{\Lambda}$ be a character. The restriction of $\varphi$ to the center $\mathrm{Z}(\Lambda)$ is a multiplicative character $\chi = \varphi_{|\mathrm{Z}(\Lambda)}: \mathrm Z(\Lambda) \to S^1$. The restriction of the unitary GNS  representation $\pi_\varphi$ to the center is given by $\pi_\varphi = \chi$.
\end{lemma}
\begin{proof}
See e.g. \cite[Lemma 7.1]{levit2024characters}.
\end{proof}

In particular, this implies that any character $\varphi \in \Ch{\Lambda}$ satisfies $\mathrm{Z}(\Lambda) \le \Pker \varphi$.

One additional construction we need is induction of traces. We follow the conventions of \cite{levit2024characters}; see \cite{kaniuth2006induced} for more information. 
Let $\Lambda_0$ be a subgroup of $\Lambda$. Recall that if $\varphi$ is a positive definite function on $\Lambda_0$, then its \emph{trivial extension} to $\Lambda$ is defined by
$$ \widetilde{\varphi} = \begin{cases}
    \varphi(g) & g\in\Lambda_0 \\
    0 & g\not\in\Lambda_0.
\end{cases}$$
The trivial extension $\widetilde{\varphi}$ is a positive definite function on $\Lambda$. In fact $\widetilde{\varphi}$ is the matrix coefficient associated to the induced unitary representation $\Ind_{\Lambda_0}^\Lambda \pi_\varphi$ acting on the Hilbert space $\cH_\varphi \otimes \ell^2(\Lambda/\Lambda_0)$ with cyclic vector $\xi_\varphi\otimes \delta_{\Lambda_0}$.

\begin{definition}
Given a finite-index subgroup  $\Lambda_0 \leq \Lambda$   and a trace $\tau \in \Tr{\Lambda_0}$,  we define the \emph{induced trace} $\Ind_{\Lambda_0}^\Lambda \tau \in \Tr{\Lambda}$ to be the function
\begin{equation}
\label{eq:formula for induction}
\Ind_{\Lambda_0}^\Lambda \tau = \frac{1}{[\Lambda:\Lambda_0]}\sum_{g \in \Lambda/\Lambda_0} \widetilde\tau^g.    
\end{equation}
\end{definition}

The fact that the induced trace $\Ind_{\Lambda_0}^\Lambda \tau$ is indeed 
normalized, positive definite and conjugation invariant follows directly from Equation (\ref{eq:formula for induction}).

\begin{proposition}
\label{prop:prop_of_Ind_res}
Let $\Lambda_0 \leq \Lambda$ be a finite-index subgroup. Let  $\tau \in \Tr{\Lambda_0}$ and $ \varphi \in \Tr{\Lambda}$ be traces.   
\begin{enumerate}
    \item If $\tau$ is weakly mixing then so is $\Ind_{\Lambda_0}^\Lambda \tau$.
    \item If $\varphi$ is weakly mixing then so is $\varphi|_{\Lambda_0}$.
\end{enumerate}
\end{proposition}

\begin{proof}
Note that the GNS representation of an induced trace is a subrepresentation of a multiple of the induced GNS representation. Similarly, the GNS representation of a restricted trace is a subrepresentation of the restricted GNS representation. Having these facts at hand, the statement follows from the analogous statements regarding weakly mixing representations, whose verification we leave to the reader.
\end{proof}

\subsection{Generalized Bekka lemma}
The following lemma is a very helpful tool to prove that a trace vanishes on a given element. This argument will play an important role in \S\ref{sec:7} below. It is also used in \S \ref{sec:4}. We quote it in the following form from \cite[Lemma 5.10]{dogon2024characters}.

\begin{lemma}[Generalized Bekka lemma]
\label{lemma:generalized bekka lemma}
Let $\Lambda$ be a countable group and $\varphi \in \mathrm{Tr}(\Lambda)$ a trace. Fix an element $g \in \Lambda$. If there exist a sequence of elements $x_n \in \Lambda$ such that $\lim_{n\to\infty} \varphi([g, x_n]^{-1} [g, x_m]) = 0$ for every $m \in \mathbb{N}$ then $\varphi(g) = 0$.
\end{lemma}

Peterson and Thom used a variant of this idea in \cite[Lemma 2.3]{PetersonThom}. As was already mentioned, we use the formulation and proof from \cite{dogon2024characters}, which generalizes the idea of Bekka in \cite[Lemma~16]{BekkaCharRig}.

\section{Arithmetic lattices}\label{sec:3}

We introduce arithmetic (and $S$-arithmetic) groups, following the  terminology of Margulis in \cite[\S I.3.1]{Margulis}. 

Let $F$ be a global field. We denote by $\mathscr{R}$ the set of all equivalence classes of  valuations on the field $F$, and by $\mathscr{R}_\infty \subset \mathscr{R}$ the subset of Archimedean ones. For each  $v \in \mathscr{R}$ let $F_v$ be  the local field obtained by completing $F$ with respect to $v$.

Fix a finite set  of valuations  $S \subset \mathscr{R}$ satisfying $\mathscr{R}_\infty \subset S$. An element $x \in F$ is called \emph{$S$-integral} if $\abs{x}_v \leq 1$ for every $v \in \mathscr{R}\backslash S$. We denote by $F(S)$ the ring of $S$-integral elements of $F$. 

Let $\mathbf{H}$ be an algebraic $F$-group, equipped with an arbitrary $F$-embedding into some matrix group $\mathbf{GL}_m$ for $m \in \N$. The group of $F(S)$-points of $\mathbf{H}$ with respect to this embedding is $\mathbf{H}(F(S)) = \mathbf{H}(F) \cap \mathbf{GL}_m(F(S))$. More generally, the \emph{(principal) congruence subgroup} associated to a given non-zero ideal $\mathcal{I} \lhd F(S)$ is the subgroup of $\mathbf{H}(F(S))$ consisting of elements congruent to identity matrix $\mathrm{Id}_m$ modulo $\mathcal{I}$. It is denoted $\mathbf{H}(\mathcal{I})$. The group $\mathbf{H}(F(S))$ as well as other congruence subgroups depend on the chosen embedding only up to commensurability \cite[Theorem~I.3.1.1.iv]{Margulis}. 
This leads  to the following definition. 

\begin{definition}
\label{def:arithmetic subgroup}
    A subgroup of $\mathbf{H}(F)$ is called \emph{$S$-arithmetic} if it is commensurable to a group of the form $\mathbf{H}(F(S))$ given by some $F$-embedding of $\mathbf{H}$ into $\mathbf{GL}_m$.
\end{definition}
\begin{definition}
A \textit{semisimple group} is a group of the form $H=\prod_{i=1}^n \mathbf{H}_i(k_i)$, where each $k_i$ is a local field and each $\mathbf{H}_i$ is a semisimple algebraic $k_i$-group. The \textit{rank} of a semisimple group is  $\mathrm{rank}(H) = \sum_{i=1}^n \text{rank}_{k_i}(\mathbf{H}_i)$. 
\end{definition}
Apriori the local fields $k_i$ can be of different characteristics, but in that case $H$ will not admit irreducible lattices. 
It is a classical result due to Borel and Harish-Chandra that arithmetic subgroups are lattices \cite[Theorem~3.1.2]{Margulis}.

\begin{theorem}[Borel--Harish-Chandra]
    Assume that the algebraic $F$-group $\mathbf{H}$ is reductive and has no $F$-rational characters. Then any $S$-arithmetic subgroup of $\mathbf{H}(F)$ is a lattice in the semisimple group $\prod_{v \in S} \mathbf{H}(F_v)$. An $S$-arithmetic lattice is uniform if and only if the group $\mathbf{H}$ is $F$-anisotropic.
\end{theorem}

We now state the celebrated Margulis arithmeticity theorem \cite[Theorem~IX.1.11]{Margulis}.

\begin{theorem}[Margulis' arithmeticity theorem]
Assume that $\text{rank}(H) \geq 2$. Every irreducible lattice in $H$ is $S$-arithmetic (for some global field $F$,  set of valuations $S$ and algebraic $F$-group $\mathbf H$;  possibly up to passing to a finite central extension).
\end{theorem}

We are now going to introduce some standing assumptions and notations to be used throughout the paper. 

\subsection{The Standing Assumptions}
\label{standing assumption}
Let $F$ be a global field with $\mathrm{char}(F) \neq 2$. Let $S \subset \mathscr{R}$ be some finite set of valuations of $F$ such that $\mathscr{R}_\infty \subset S$ and  $|S| \ge 2$. Denote by $R = F(S)$ the ring  of $S$-integral elements of $F$. Let $\mathbf{G}$ be a connected,  simply-connected and absolutely almost simple algebraic $F$-group with $\mathrm{rank}_F(\mathbf G) = 1$. We equip $\mathbf{G}$ with a fixed $F$-embedding into $\mathbf{GL}_m$ for some $m \in \N$. We define $\Gamma$ to be the $S$-arithmetic group $\mathbf{G}(R)$, so that $\Gamma$ is a non-uniform irreducible lattice in the product semisimple group $\prod_{v \in S} \mathbf{G}(F_v)$.

Whenever the symbols $\mathbf{G}$, $F$, $S$, $R$ and $\Gamma$ are to be understood in the sense of our Standing Assumptions,  we will emphasize this. The relevance of those  assumptions to Theorem \ref{main_thm}  is justified by the following somewhat technical proposition.

\begin{proposition}\label{prop:notationIsGood}
Let $H$ be a semisimple group with $\mathrm{rank}(H) \ge 2$ so that $H$ admits a simple rank-one factor. Let $\Lambda$ be a non-uniform irreducible lattice in $H$. Then there exist $\mathbf{G}$, $F$, $S$, $R$ and $\Gamma$ as in the Standing Assumptions, and a group $\Delta$ commensurable to  $\Lambda$, such that $\Gamma$ is a central extension of $\Delta$.
\end{proposition}

The particular standing assumption  $\text{char}(F) \ne 2$ is not needed here. 

\begin{proof}[Proof of Proposition \ref{prop:notationIsGood}]
The semisimple group $H$ is of the form $H=\prod_{i=1}^n \mathbf{H}_i(k_i)$ for some local fields $k_i$ and semisimple algebraic $k_i$-groups $\mathbf{H}_i$. The Margulis arithmeticity theorem shows that there is a global field $F$, a finite set  $S \subset \mathscr{R}$ of valuations on $F$ with $\mathscr{R}_\infty \subset S$ and a connected, simply-connected absolutely almost simple algebraic $F$-group $\mathbf{G}$ such that:
    \begin{enumerate}
        \item There is a continuous homomorphism $f  : \prod_{v\in S}\mathbf{G}(F_v) \to H$ with $\ker f$ central and with $\mathrm{Im} f$ closed, normal and cocompact in $H$.
        \item The $S$-arithmetic group $\Gamma = \mathbf{G}(F(S))$  satisfies that its image $\Delta = f(\Gamma)$ is commensurable with $\Lambda$.        
\end{enumerate}
The fact that the algebraic $F$-group $\mathbf{G}$ can without loss of generality be taken simply-connected is established in \cite[Remark IX.1.6.i]{Margulis}.
Since the lattice $\Lambda$ is non-uniform, the lattice $\Gamma$ is non-uniform as well. Hence the group $\mathbf{G}$ is $F$-isotropic. We remark that in the case of uniform lattices, one would need to take into consideration potential Archimedean valuations $v \in \mathscr{R}_\infty$ for which $\mathbf{G}(F_v)$ is compact but is not a factor of $H$. In the non-uniform case this nuance can be ignored.
    
The assumption that some factor of the semisimple group $H$ has rank one implies that $\mathrm{rank}_{F_v}(\mathbf{G}) = 1$  for some valuation $v \in S$ \cite[Remark IX.1.3.vii]{Margulis}. This forces the condition $\mathrm{rank}_F(\mathbf{G}) \le 1$. As $\mathbf{G}$ is $F$-isotropic we get $\mathrm{rank}_F(\mathbf{G}) = 1$. In addition, it must be the case  that $\abs{S} \geq 2$ because the semisimple group $H$ admits at least two factors. We conclude that the objects  $\mathbf{G}$, $F$, $S$, $R$ and $\Gamma$ satisfy all the Standing Assumptions.
\end{proof}

\begin{remark}\label{rem:Dirichlet}
The group of units of the ring $R = F(S)$ of $S$-integral elements of the global field $F$ is an abelian group whose rank is $|S|-1$. This is the so called Dirichlet--Hasse--Chevalley unit theorem \cite[Theorem~5.3.10]{WeissANT}. In other words, the assumption $\abs{S} \geq 2$ is equivalent to saying that the ring $R$ admits an infinite order unit. For example, if $F$ is a number field and $S$ consists of its Archimedean valuations, then $R = F(S)$ is the ring of algebraic integers of $F$ and the above claim is the classical unit theorem of Dirichlet. In this case, the assumption $|S| \ge 2$ says that $F$ is neither the field of rational numbers nor an imaginary quadratic field. Dirichlet's unit theorem plays an essential role for us. By Margulis' arithmeticity theorem, the manifestation of the assumption $|S| \ge 2$ in our context is that our arithmetic lattice is an irreducible lattice in a non-trivial product. 
\end{remark}

\begin{remark}
\label{remark:abelian or two step nilpotent}
The $F$-group $\mathbf{G}$ has $\mathrm{rank}_F(\mathbf{G})=1$. Hence its relative root system is either of type $\text{A}_1$ or of type $\text{BC}_1$. Denote by $\mathbf{U}$ the unipotent radical of a minimal $F$-parabolic $\mathbf{P}$. Then $\mathbf{U}$ is either abelian (if the root system is $\text{A}_1$), or two-step nilpotent (if the root system is $\text{BC}_1$). 
\end{remark}

\begin{remark}
Assume that $\mathrm{char}(F)=0$. In that case, the lattices for which character rigidity was not previously known to follow from  existing works  arise as follows: there exists an Archimedean valuation  $v \in S$ such that $\mathbf{G}(F_v)$ is locally isomorphic to either $\mathrm{SO}(n,1)$ or $\mathrm{SU}(n,1)$. This forces the $F$-group $\mathbf{G}$ to come from a restrictive list of algebraic groups, and the arithmetic group $\Gamma$ has one of  two possible descriptions given in Appendix \ref{appen}.
\end{remark}

\begin{remark}
Our terminology of \enquote{semisimple groups} is not standard; there does not seem to be any well-established such term. In Margulis' book \cite{Margulis} no such umbrella term is introduced. Several authors have introduced other closely related variants, see e.g. \cite{levit2020benjamini, gelander2018invariant,gelander2022effective,bader2024spectral}. The  terminology we use here is quite flexible (compared to, say,  \enquote{standard semisimple groups} of \cite{bader2024spectral}). This leeway is justified by Proposition \ref{prop:notationIsGood}.
\end{remark}

\subsection{Arithmetic groups under $F$-morphisms}
When working with arithmetic groups, we will sometimes need to control  different embeddings more carefully than by just saying that the resulting groups are commensurable. To this end, the following lemmas will be useful.

\begin{lemma}[{\cite[Lemma~I.3.1.1.(ii)]{Margulis}}]\label{lemma:congruence under F-morphisms}
    Let $\mathbf{H}_1 \subseteq \mathbf{GL}_m$ and $\mathbf{H}_2 \subseteq \mathbf{GL}_n$ be two algebraic $F$-groups and $f: \mathbf{H}_1 \rightarrow \mathbf{H}_2$ be an $F$-algebraic map. If either
    \begin{itemize}
        \item $f$ is a homomorphism, or
        \item $f$ takes the identity element of $\mathbf{H}_1$ to the identity element of $\mathbf{H}_2$ and $\mathbf{H}_2 \subseteq \mathbf{SL}_n$,
    \end{itemize}
    then for every non-zero ideal $\mathcal I \lhd R$ there exists a non-zero ideal $\mathcal J \lhd R$ such that $f(\mathbf{H}_1(\mathcal J)) \subseteq \mathbf{H}_2(\mathcal I)$. 
\end{lemma}

Recall that for an $F$-split torus $\mathbf{A}$, the  group $\mathbf{A}(R)$ is infinite if and only if $\abs{S}\geq 2$ (Remark \ref{rem:Dirichlet}).

\begin{lemma}
\label{lemma:arithmetic torus}
    Let $\mathbf A$ be an $F$-split torus. Denote by $X(\mathbf{A)}$ the group of characters of  $\mathbf{A}$. Then every $\alpha \in X(\mathbf A)$ satisfies $\alpha(\mathbf A(R)) \subset R$.
\end{lemma}
\begin{proof}
We regard the torus $\mathbf{A}$ with respect to some embedding into $\mathbf{GL}_m$ for $m \in \N$. Since $\mathbf{A}$ is $F$-split, there exists an $F$-isomorphism $\Psi$ of $\mathbf{A}$ onto the diagonal matrices $\mathbf{D}_n \le \mathbf{GL}_n$ for some $n \in \N$. We claim that $\Psi(\mathbf{A}(R)) \subset \mathbf{D}_n(R)$. Consider an element $x \in \mathbf{A}(R)$. The two groups $\Psi(\mathbf{A}(R))$ and $\mathbf{D}_n(R)$ are commensurable by \cite[Lemma I.3.1.1]{Margulis}. Hence there is some power $k \in \N$ such that $\Psi(x)^k \in \mathbf{D}_n(R)$. This means that $\Psi(x)$ is a diagonal matrix with entries from the field $F$ whose $k$-th power belong to the ring $R$. Since $R$ is integrally closed, this implies that $\Psi(x) \in \mathbf{D}_n(R)$ in the first place, as required.
The claim implies that every character $\beta \in X_F(\mathbf D_n)$ 
satisfies $\beta(\Psi(\mathbf A(R))) \subset  R$. However, every character  $\alpha \in X_F(\mathbf{A})$ is of the form $\alpha = \beta \circ \Psi$ for some character $\beta \in X_F(\mathbf D_n)$. The desired conclusion follows.
\end{proof}

Roughly speaking, the following lemma shows that taking the quotient of a unipotent algebraic group by a normal subgroup, essentially induces a surjection on arithmetic groups. We include a proof of this lemma for lack of good reference.

\begin{lemma}
\label{label:modding out unipotents}
Let $\mathbf{G}$ be a connected unipotent algebraic $F$-group and  $\mathbf H$ be a connected $F$-split normal subgroup of $\mathbf {G}$. The quotient $\mathbf{G}/ \mathbf{H}$ is an algebraic $F$-group. Let  $\pi : \mathbf{G} \to \mathbf{G}/ \mathbf{H}$ be the $F$-rational quotient map. For each non-zero ideal $\mathcal{I} \lhd R$ there is a non-zero ideal $\mathcal{J} \lhd R$ such that $(\mathbf{G}/\mathbf{H})(\mathcal J) \subset \pi(\mathbf G(\mathcal I))$.
\end{lemma}
\begin{proof}
The quotient $\mathbf{G}/ \mathbf{H}$ is indeed an algebraic $F$-group and the quotient map is $F$-rational  \cite[Theorem 6.8]{Borel}. There is an $F$-isomorphism of varieties $\Psi : \mathbf G/ \mathbf H \times \mathbf H \to \mathbf G$, such that the composition of $\Psi^{-1}$ with the projection to the first coordinate is the quotient map $\pi$  \cite[Theorem~14.2.6]{Springer}. 
The map $\Psi$ is not a homomorphism, but it does take the identity to the identity,  and its range is the unipotent group $\mathbf{G}$  embedded in $\mathbf{SL}_n$. Hence Lemma \ref{lemma:congruence under F-morphisms} is applicable with respect to the map $\Psi$. Specifically, 
given any non-zero ideal $\mathcal{I} \lhd R$ we may apply Lemma \ref{lemma:congruence under F-morphisms} and find a non-zero ideal $\mathcal J \lhd R$ with $$\Psi\left( (\mathbf G/ \mathbf H \times \mathbf H)(\mathcal{J})\right) \subset \mathbf{G}(\mathcal{I}).$$
 Applying $\Psi^{-1}$ to both sides and then projecting to $\mathbf G/\mathbf H$ implies the desired conclusion.
\end{proof}

\section{Commensurability and central extensions}\label{sec:4}

Recall that a countable group $\Lambda$ is called \emph{ICC (infinite conjugacy classes)} if the conjugacy class of every non-trivial element of $\Lambda$ is infinite.

In this section we establish that character rigidity (in the strict sense of Definition \ref{def:character rigidity}) is a commensurability invariant for ICC groups admitting at most countably many finite-dimensional unitary representations. We also show that character rigidity passes to (possibly infinite) central extensions, under the same assumption on finite-dimensional representations. The latter observation will allow us we deduce Theorem \ref{inf_ctr} for lattices in higher-rank semisimple Lie groups with infinite center.

Having at most countably many finite-dimensional unitary representations is useful due to the following observation, whose proof is immediate from the definitions: 
\begin{lemma}
\label{lemma:applying finitely many}
Let $\Lambda$ be a character rigid group admitting at most countably many finite-dimensional unitary representations. Then any weakly mixing trace on $\Lambda$ vanishes off the center $\mathrm{Z}(\Lambda)$.
\end{lemma}

On the other hand, if a countable group has uncountably many finite-dimensional unitary representations, it will admit weakly mixing traces whose Bochner transform is supported on finite-dimensional characters (without any atoms).

It is useful to keep in mind that the Dirac function $\delta_{e}$ is always a trace on any countable group $\Lambda$. It is a character if and only if $\Lambda$ is ICC \cite[Proposition 7.A.1]{BdlH}.

\begin{proposition} \label{prop:char_rig_commensurability}
Let $\Lambda$ be an ICC countable group having at most countably many finite-dimensional unitary representations. Let  $\Lambda_0$ be a finite-index subgroup of $\Lambda$. Then $\Lambda$ is character rigid if and only $\Lambda_0$ is.
\end{proposition}
\begin{proof}
In one direction, we assume that $\Lambda$ is character rigid and show that $\Lambda_0$ is character rigid as well.
Let $\tau \in \Ch{\Lambda_0}$ be a weakly mixing character.
As such, $|\tau|^2$ is a weakly mixing trace.
By Proposition \ref{prop:prop_of_Ind_res}, the induced trace $\varphi = \Ind_{\Lambda_0}^\Lambda |\tau|^2$ is also weakly mixing.
Lemma \ref{lemma:applying finitely many} implies that $\varphi$ is the Dirac character on $\Lambda$. Now, the induced trace $\varphi$ is a normalized sum of finitely many non-negative functions of the form $\widetilde{|\tau|^2}^g$.  It follows that each one of these summands must in itself be equal to the Dirac character. In particular  $|\tau|^2 =\delta_e$ so that $\tau = \delta_e$ as required.

In the opposite direction, we assume that $\Lambda_0$ is character rigid and show that $\Lambda$ is character rigid as well. The subgroup $\Lambda_0$ is  ICC, being a finite-index subgroup of  an ICC group. In particular $\mathrm{Z}(\Lambda) = \mathrm{Z}(\Lambda_0) = \{e\}$.
Let $\varphi$ be a weakly mixing character of $\Lambda$. Hence the restriction $\varphi|_{\Lambda_0}$ is a weakly mixing trace by Proposition \ref{prop:prop_of_Ind_res}. The subgroup $\Lambda_0$ is center-free, and has at most countably many finite-dimensional unitary representations by Proposition \ref{finitly many reps comensurability} below. Hence  $\varphi|_{\Lambda_0}$ is the Dirac character $\delta_e$ by Lemma \ref{lemma:applying finitely many}. By the ICC property, for every non-trivial element $g \in \Lambda$ there are infinitely many elements $x_n \in \Lambda$ such that the conjugates $g^{x_n}$ are pairwise distinct. Up to passing to a subsequence, we may assume without loss of generality that the pairwise distinct commutators $\left[g,x_n\right]$ all belong to the same coset of $\Lambda/\Lambda_0$. We conclude that $\varphi(g) = 0$ by applying the generalized Bekka lemma (see Lemma \ref{lemma:generalized bekka lemma}) with respect to this data.
\end{proof}

\begin{remark}
Proposition \ref{prop:char_rig_commensurability} holds true for any  ICC group, without the extra assumption on finite-dimensional representations. However, the proof becomes more technical, and we will only need to apply this proposition to higher-rank lattices, which do satisfy this extra assumption (see \cite[Proposition~7.1]{BBHP}).
\end{remark}
\begin{proposition} \label{prop:char_rig_extensions}
Let $\Lambda$ be a countable group with center $C = \mathrm{Z}(\Lambda)$. Assume that the quotient group $\Lambda/C$ is center-free and has at most countably many finite-dimensional unitary representations. Then $\Lambda$ is character rigid if and only if $\Lambda/C$ is.
\end{proposition}
\begin{proof}
    Assume that the quotient group $\Lambda/C$ is character rigid. We show that $\Lambda$ itself is character rigid.
    Let $\varphi \in \Ch{\Lambda}$ be a weakly mixing character. Recall that
    $C \leq \Pker \varphi$ by Lemma \ref{lemma:schur} and the discussion following it. Hence the trace $|\varphi|^2$ factors to a weakly mixing trace of the quotient $\Lambda/C$.
    Character rigidity and Lemma \ref{lemma:applying finitely many} implies that $|\varphi|^2$ factors to the Dirac trace of $\Lambda/C$. Equivalently $|\varphi|^2$ vanishes outside of $C$. 
    Consequently, the character $\varphi$ also vanishes outside of $C$, as required.
    
    Conversely, assume that $\Lambda$ is character rigid. We show that the quotient $\Lambda/C$ is character rigid.
    Indeed, if $\varphi  \in \Ch{\Lambda/C}$ is a weakly mixing character, then it is also a weakly mixing character when pulled back to $\Lambda$.
    By character rigidity for $\Lambda$, the pullback of $\varphi$ must vanish outside of $C$, so that $\varphi$ is the Dirac trace of $\Lambda/C$. This completes the proof.
\end{proof}

\begin{corollary}\label{cor:char_rig_commensuability_nonICC}
Let $H$ be a higher-rank semisimple group and $\Lambda_1,\Lambda_2$ be two commensurable lattices in $H$. Then  $\Lambda_1$ is character rigid if and only if $\Lambda_2$ is.
\end{corollary}
\begin{proof}
Consider the adjoint semisimple group $\mathrm{Ad}(H) = H/\mathrm{Z}(H)$. 
The projections $\mathrm{Ad}(\Lambda_1)$ and $\mathrm{Ad}(\Lambda_2)$ of the two lattices $\Lambda_1$ and $\Lambda_2$ to $\text{Ad}(H)$ are still commensurable. 
Further $\mathrm{Ad}(\Lambda_1)$ and $\mathrm{Ad}(\Lambda_2)$ are ICC groups \cite[Lemma 5.4]{bader2024spectral}. Certainly $\Lambda_1$ and $\Lambda_2$ are central extensions of the center-free groups $\mathrm{Ad}(\Lambda_1)$ and $\mathrm{Ad}(\Lambda_2)$. The result now follows by putting together Propositions \ref{prop:char_rig_commensurability} and \ref{prop:char_rig_extensions}. We note that all lattices in question have at most countably many finite-dimensional unitary representations by \cite[Proposition~7.1]{BBHP}.
\end{proof}

\section{Characters of arithmetic solvable groups}\label{sec:5}

We study the character theory of certain arithmetic solvable groups, building towards the proof of our key Theorem \ref{theorem:on characters of G admitting AU and AV}. The first goal is to characterize congruence characters of a unipotent arithmetic group $U$ in terms of their orbits under the action of a certain semisimple element $a$ (Proposition \ref{prop:rational if and only of finite orbit}). The second goal is to classify characters of the semidirect product  $M = \langle a \rangle \ltimes U$ (Proposition \ref{prop:characters of AU}). This information is then used to study traces of this semidirect product $M$ whose restriction to the cyclic subgroup $\langle a \rangle$ of a split torus is mixing.

In the current section we will freely use the Standing Assumptions \S\ref{standing assumption}.  In particular, recall that the field $F$ and the set of valuations $S$ are chosen such that the ring $R=F(S)$ has an infinite order unit.  See \S\ref{sec:3} for more details on arithmetic and congruence subgroups.

\subsection{Congruence characters}

Let $\mathbf{U}$ be a connected algebraic $F$-group and $\mathbf{U}(R)$ be an arithmetic group.

\begin{definition}
\label{def:congruence characters}
A trace $\varphi \in \Tr{\mathbf{U}(R)}$ is called a \emph{congruence trace} if its kernel $\ker \varphi$ contains the congruence subgroup $\mathbf{U}(\mathcal I)$ for some non-zero ideal $\mathcal{I} \lhd R$. A \emph{congruence character} is a congruence trace which is a character.

A unitary representation $\pi$ of $\mathbf{U}(R)$ is called a \emph{congruence representation} if its kernel $\ker\pi$ contains the congruence subgroup $\mathbf{U}(\mathcal I)$ for some non-zero ideal $\mathcal{I} \lhd R$.
\end{definition}

\begin{remark}
The GNS representation of a congruence trace is an example of a congruence representation.    
\end{remark}

We begin with the following useful structural lemma regarding an algebraic group endowed with an action of a split torus. For an algebraic $F$-group $\mathbf{U}$ and an $F$-split torus $\mathbf{S}$ acting on it algebraically, we denote by $\Phi(\mathbf{S},\mathbf{U})$ the set of weights of $\mathbf{S}$ appearing in the Lie algebra $\mathrm{Lie}(\mathbf{U})$. The reader is referred to \cite[\S I.3]{Borel} for a discussion on Lie algebras of  algebraic groups.

\begin{lemma}\label{lemma: central series}
Assume that the connected algebraic $F$-group $\mathbf{U}$ is $F$-split unipotent. Let  $\mathbf{S}$ be an $F$-split torus acting algebraically on $\mathbf{U}$. If $0 \notin \Phi(\mathbf{S},\mathbf{U})$ then there exists an $\mathbf{S}$-invariant central series $\{e\} = \mathbf{U}_0 \subseteq \mathbf{U}_1 \cdots \subseteq \mathbf{U}_n=\mathbf{U}$ such that the subquotients $\mathbf{U}_i / \mathbf{U}_{i-1}$ are $\mathbf{S}$-equivariantly $F$-isomorphic  to a vector group with a scalar action. 
\end{lemma}
\begin{proof}

By the assumption that $0 \notin \Phi(\mathbf{S},\mathbf{U})$, we may find a central $F$-subgroup $\mathbf{U}_1$ that is $\mathbf{S}$-equivariantly $F$-isomorphic to a vector group with a diagonal linear action of $\mathbf{S}$, as explained in the proof of \cite[Theorem~3.3.11]{CGP}. By passing to a smaller subgroup, we may further assume without loss of generality that $\mathbf{S}$ acts on $\mathbf{U}_1$ with a single weight. 

Next, we consider the quotient $F$-group $\mathbf{U} / \mathbf{U}_1$ (see \cite[Theorem~6.8]{Borel} for the construction of an $F$-group structure on quotients). Its Lie algebra can be identified $\mathbf{S}$-equivariantly with a quotient of the Lie algebra of $\mathbf{U}$. It follows that $0 \notin \Phi(\mathbf{S},\mathbf{U}/
\mathbf{U}_1)$. We can now repeat the above argument.
Proceeding by induction on the dimension, we build the entire central series $\mathbf{U}_i$ as required.
\end{proof}

Consider an algebraic semidirect product group $\mathbf{A} \ltimes \mathbf{U}$, where $\mathbf{A}$ is an $F$-split torus and $\mathbf{U}$ is $F$-split unipotent. We would like to characterize congruence characters of $\mathbf{U}(R)$ in terms of the $\mathbf{A}(R)$-action on $\mathrm{Ch}(\mathbf{U}(R))$. Assume that $0 \notin \Phi(\mathbf A, \mathbf U)$ and equip $\mathbf U$ with a central series $\mathbf U_i$ as in Lemma \ref{lemma: central series}. Fix an element $a \in \mathbf{A}(R)$ such that no non-trivial power of $a$ lies in the kernel of any weight from $\Phi(\mathbf A,\mathbf U)$. For example, any element $a \in \mathbf A$ generating a Zariski-dense subgroup will have this property. Denote  $T = \left<a\right>$ so that $T$ is a cyclic subgroup of $\mathbf A(R)$.

\begin{remark} \label{good a exists}
Under our Standing Assumptions \S\ref{standing assumption} the torus $\mathbf{A}$ is one-dimensional, so any infinite order element $a \in \mathbf A(R)$ generates a Zariski-dense subgroup (see Remark \ref{rem:Dirichlet}). Generally speaking, an element generating a Zariski-dense subgroup may not exist even if $|S| \ge 2$. However, an element $a \in \mathbf A(R)$ whose powers do not lie in the kernel of any weight does always exist if $\abs{S} \geq 2$.
Indeed, since the collection of weights $\Phi(\mathbf A, \mathbf U)$ is finite, there exists a one-dimensional sub-torus $\mathbf S$ not contained in the kernel of any weight. Since $\mathbf S$ is connected and one-dimensional, any proper Zariski-closed subgroup of $\mathbf S$ is finite. In particular $|\mathbf S \cap \ker \omega| < \infty$  for every $\omega \in \Phi(\mathbf A, \mathbf U)$. Any infinite order element  $a \in \mathbf S (R)$ is as required.
 \end{remark}

\begin{proposition}
\label{prop:rational if and only of finite orbit}
Let $\varphi \in \Ch{\mathbf{U}(R)}$ be a character. Let $i(\varphi) \in \N \cup \{0\}$ be the maximal index such that the restriction $\varphi_{|\mathbf U_{i(\varphi) }(R)}$ is a congruence trace, and let $\mathcal{I} \lhd R$ be such that $\mathbf U_{i(\varphi) }(\mathcal{I}) \le \ker \varphi$. 
\begin{enumerate}
    \item The following three conditions are equivalent:
    \begin{enumerate}
    \item $\mathbf U_{i(\varphi)} = \mathbf U$, 
    \item $\varphi$ is a congruence character, and
    \item the orbit of $\varphi$ under the dual conjugation action of $T$ is finite.
    \end{enumerate}
    \item If $\mathbf U_{i(\varphi)} \lneq \mathbf U$ then the restriction $\varphi_{|\mathbf{U}_{i(\varphi) +1}(\mathcal{I})}$ is a multiplicative character whose  orbit under the dual conjugation action of $T$ is infinite.
\end{enumerate}
\end{proposition}

\begin{proof}
First consider part (1) of the statement. The implication $(a) \Rightarrow (b)$ is immediate from the definitions.

Let us consider the implication $(b) \Rightarrow (c)$. Assume that $\varphi$ is a congruence character. There exists a non-zero ideal $\mathcal{J} \lhd R$ such that $\varphi$ factors through the finite quotient group $\mathbf U(R) /  \mathbf U(\mathcal{J})$. Hence $\varphi$ is fixed by a certain power of the element $a$, that fixes pointwise the entire finite Thoma dual of this finite quotient.

The main part of the proposition is the implication $(c) \Rightarrow (a)$. Assume towards contradiction that the orbit of $\varphi$ under the dual conjugation action of $T$ is finite but $\varphi$ is not a congruence character. In particular, this means that $\mathbf U_{i(\varphi)} \lneq \mathbf U$. To ease our notations, we set $i = i(\varphi)$.

The subquotient $\mathbf{U}_{i+1} (\mathcal I) / \mathbf{U}_{i}(\mathcal{I})$ is central in the quotient group $\mathbf{U}( R) / \mathbf{U}_i(\mathcal I)$. Indeed
$$ \left[\mathbf{U}( R), \mathbf{U}_{i+1} (\mathcal I) \right] \subset \mathbf {U}_i(R) \cap \mathbf U(\mathcal{I}) = \mathbf U_i(\mathcal{I}).$$
By Schur's lemma for characters (Lemma \ref{lemma:schur}), the restriction of the character $\varphi$ to $\mathbf{U}_{i+1} (\mathcal I)$ is multiplicative. By our assumption, there exists some power $n \in \N$ such that conjugation by the element $a^n \in T$ fixes the restriction $\varphi_{|\mathbf{U}_{i+1} (\mathcal{I})} \in \Ch{\mathbf{U}_{i+1}(\mathcal{I})}$.

There is an $\mathbf A$-equivariant $F$-isomorphism $\Psi : \mathbf U_{i+1}/\mathbf U_i \to \mathbf V$ where $\mathbf V$ is some additive vector $F$-group on which $\mathbf{A}$ acts via a single weight $\omega \in \Phi(\mathbf A,\mathbf V) \subset \Phi(\mathbf A, \mathbf U)$; see Lemma \ref{lemma: central series}. Hence, the element $a^n$ acts on the vector group $\mathbf V$ linearly by some scalar $\beta \in F$. This scalar must satisfy $\beta \in R $ by Lemma \ref{lemma:arithmetic torus} as well as $\beta \neq 1$ by our assumption on the element $a$. 

Consider the $F$-rational quotient map $\overline{\Psi}:\mathbf U_{i+1} \to \mathbf V$. According to Lemma \ref{label:modding out unipotents} there is some non-zero ideal $\mathcal L \lhd R$ such that $\mathbf V(\mathcal{L}) \subset \overline{\Psi}(\mathbf U_{i+1} (\mathcal I))$. The character $\varphi$ descends to a multiplicative character $\psi$ of the group $\mathbf V(\mathcal{L})$. We now claim that $\mathbf V((\beta-1)\mathcal L) \subset \ker \psi$. Recall that $\varphi^{a^n} = \varphi$. Hence, for every element $v \in \mathbf V(\mathcal{L})$ we have
$$ \psi(v^{a^n}) = \psi(v) \quad \Rightarrow \quad \psi(\beta v) \psi(v)^{-1} = 1\quad \Rightarrow \quad \psi((\beta-1)v) = 1.$$
This establishes the claim.  Lemma \ref{lemma:congruence under F-morphisms} allows us to find  a non-zero ideal $\mathcal{J} \lhd R$ such that $\overline{\Psi}(\mathbf U_{i+1}(\mathcal J)) \subset \mathbf  V((\beta-1)\mathcal L)$. In other words $\mathbf{U}_{i+1}(\mathcal{J}) \le \ker \varphi$, namely the restriction of $\varphi$ to the subgroup $\mathbf U_{i+1}(R)$ is a congruence trace. This stands in contradiction to the choice of the index $i = i(\varphi)$.

Lastly, we address part (2) of the statement. Assume that $\mathbf{U}_{i(\varphi)} \lneq \mathbf U$. The above proof of the implication $(c)\Rightarrow(a)$ shows that the restriction of $\varphi$ to $\mathbf U_{i+1}(\mathcal{I})$ is a multiplicative character. Further, this restriction must have an infinite orbit under the dual action of $T$, for otherwise $\varphi$ would be a congruence character of $\mathbf{U}_{i+1}$.
\end{proof}

\subsection{Induced characters}

We add some notations in order to emphasize arithmetic rather than algebraic groups. Specifically\footnote{The same convention will apply with respect to other letters denoting algebraic groups, namely $A = \mathbf A(R), U = \mathbf{U}(R)$, etc.}, given an algebraic $F$-group $\mathbf{G}$, we will denote by $G$ the corresponding arithmetic group $\mathbf{G}(R)$, and by $G(\mathcal I)$ the congruence subgroup $\mathbf{G}(\mathcal{I})$ at each non-zero ideal $\mathcal I \subseteq R$.

Let   $\mathbf{U}$ be a connected $F$-split unipotent algebraic group and $\mathbf{A}$ be  an $F$-split torus acting on $\mathbf{U}$ algebraically such that $0 \notin \Phi(\mathbf A,\mathbf U)$. We fix an element $a \in A = \mathbf{A}(R)$, such that no power of $a$ lies in the kernel of any weight  $\alpha \in \Phi(\mathbf A,\mathbf U)$. Denote $T = \left<a\right>$ so that $T \le A$. Consider the semidirect product group $M = T \ltimes U$ where $U = \mathbf U(R)$.  Our goal is to study characters of $M$.

\begin{proposition}
\label{prop:characters of AU}
 Let $\varphi \in \Ch{M}$ be a character, where $M = T \ltimes U$. Then either
\begin{enumerate}[label=(\alph*)]
\item $\varphi(a^i u) = 0 $ for all $i \in \Z\setminus\{0\}$ and all $u \in U$, or
\item $\varphi_{|U}$ is a congruence trace.
\end{enumerate}
These two possibilities are mutually exclusive. In particular, characters of type (a) are infinite-dimensional while characters of type (b) are finite-dimensional.
\end{proposition}

\begin{proof}
The Bochner transform of the restriction $\varphi_{|U}$ is an ergodic $T$-invariant probability measure $\mu_\varphi \in \mathrm{Prob}(\Ch{U})$ (for this fact, as well as a general discussion on relative traces, we refer the reader e.g. to \cite[\S2]{carey1984characters}). 

To each character $\psi \in \Ch{U}$ we have associated in Proposition \ref{prop:rational if and only of finite orbit} an invariant $i(\psi) \in \N \cup \{0\}$. This invariant is preserved under the dual action of $T$, in the sense that $i(\psi) = i(\psi^a)$ for all $\psi \in \Ch{U}$. By the ergodicity of the measure $\mu_\varphi$,  this invariant must be $\mu_\varphi$-almost everywhere constant, so that there is some $i_\varphi \in \N \cup \{0\}$ with $i(\psi) = i_\varphi$  for $\mu_\varphi$-almost every $\psi \in \Ch{U}$. 
Furthermore, for any 
ideal $\mathcal{I} \lhd R$, the subset of $\Ch{U}$ consisting of characters $\psi$ satisfying $U_{i_\varphi}(\mathcal{I}) \le \ker \psi$ is invariant under the dual action of $T$. Since $R$ has only countably many ideals, there exists a non-zero ideal $\mathcal{I}$ such that for $\mu_\varphi$-almost every character $\psi$ satisfies $U_{i_\varphi}(\mathcal{I}) \le \ker \psi$. Hence $U_{i_\varphi}(\mathcal{I}) \le \ker \varphi$ as well. 

Our analysis of the character $\varphi$ depends on the invariant $i_\varphi$. 
\begin{itemize}
\item 
If $U_{i_\varphi} \lneq U$ then  Proposition \ref{prop:rational if and only of finite orbit} implies that  $\mu_\varphi$-almost every character  $\psi \in \Ch{U}$ restricts to a multiplicative character of $U_{i_\varphi+1}$ whose orbit under the dual action of $T$ is infinite. In other words, the resulting $T$-action on the Pontryagin dual of $U_{i_\varphi+1}(\mathcal{I})/U_{i_\varphi}(\mathcal{I})$ equipped with the probability measure coming from Bochner transform of the restriction of $\varphi$ is essentially free.
Therefore $\varphi(a^i) = 0 $ for all $i \in \Z \setminus \{0\}$ by \cite[Lemma 3.5]{levit2024characters}. In fact, as the group $U$ acts trivially on $U_{i_\varphi+1}(\mathcal{I})/U_{i_\varphi}(\mathcal{I})$, the same applies to any element of the form $h = a^i u$ with $i \in \Z \setminus \{0\}$ and $u \in U$. We arrive at possibility (a).

\item If $U_{i_\varphi} = U$ then the restriction $\varphi_{|U}$ is a congruence trace. This corresponds to possibility (b).
\end{itemize}

It remains to verify that the two possibilities (a) and (b) are mutually exclusive. In case (a) the character $\varphi$ vanishes outside of the infinite-index subgroup $U$. In particular, its GNS representation is induced from the subgroup $U$. As such, $\varphi$ is infinite-dimensional. In case (b), we have $\left[U: U \cap \ker \varphi \right] < \infty$. Certainly $\varphi$ factors through the quotient group $K = M/\ker \varphi$. The group $K$ is cyclic-by-finite. Every character of a virtually abelian group is finite-dimensional. Hence $\varphi$ is finite-dimensional.
\end{proof}

Proposition \ref {prop:characters of AU} shows that the character space $\Ch{M}$ decomposes as a disjoint union of two Borel sets $$\Ch{M}=\Ch{M}_{\mathrm{ind}} \bigsqcup \Ch{M}_{\mathrm{cong}},$$
corresponding respectively to characters of type (a) and (b) as above. It will be convenient to introduce the notations
$$ \Tr{M}_\mathrm{ind} = \{ \varphi \in \Tr{M} \: : \: \mu_\varphi(\Ch{M}_\mathrm{ind}) = 1 \}$$
and likewise for $\Tr{M}_\mathrm{cong}$.
In light of Proposition \ref{prop:disj traces and gns}, this yields a canonical decomposition of any trace $\varphi \in \Tr{M}$ into a convex combination of two disjoint  traces:
$$\varphi = t\varphi_{\mathrm{ind}} + (1-t)\varphi_{\mathrm{cong}}; \quad \varphi_{\mathrm{ind}} \in \Tr{M}_{\mathrm{ind}}, \varphi_{\mathrm{cong}} \in \Tr{M}_{\mathrm{cong}},t\in\left[0,1\right].$$

 \begin{remark}
A trace $\varphi \in \Tr{M}_\mathrm{cong}$ does not necessarily restrict to a congruence trace on $U$. Rather, this restriction is  a convex combination of at most countably many congruence characters of $U$. 
 \end{remark}

\begin{proposition} \label{prop: mixing is uniform in U}
Let $\varphi \in \Tr{M}$ be a trace. If the restriction of its associated GNS representation $\pi_\varphi$ to the subgroup $T$ is mixing then
$$ \lim_{i \to \infty} \sup_{u \in U} \;\abs{\varphi \, (a^i u)} = 0.$$
\end{proposition}

\begin{proof}
Consider the decomposition $\varphi = t\varphi_{\mathrm{cong}} + (1-t)\varphi_{\mathrm{ind}}$ as a convex combitation of two disjoint traces, for some $t \in \left[0,1\right]$. 
Proposition \ref{prop:characters of AU} says that the trace $\varphi_{\mathrm{ind}}$ vanishes on $M\backslash U$, so it suffices to prove the statement for the trace $\varphi_{\mathrm{cong}}$. 

Let $(\mathcal{H}_\varphi, \pi_\varphi,\xi_\varphi)$ be the GNS-triple assocated to $\varphi$. According to Proposition \ref{prop:disj traces and gns} there is a Hilbert subspace 
$\mathcal{H}_\mathrm{cong} \le \mathcal{H}_\varphi$ corresponding to the trace $\varphi_{\mathrm{cong}}$ and a corresponding cyclic unit vector $\xi_{\mathrm{cong}} \in \mathcal{H}_\mathrm{cong}$ such that $\varphi_{\mathrm{cong}}(g) = \left<\pi_\varphi(g)\xi_{\mathrm{cong}},\xi_{\mathrm{cong}}\right>$ for all $g \in \Lambda$. Since the ring $R$ has countably many non-zero ideals, the Hilbert space $\mathcal{H}_\mathrm{cong}$ can be represented as the closure of an increasing union of invariant subspaces $\mathcal{K}_n$ such that $U(\mathcal{I}_n)\subseteq \ker(\pi_\varphi|_{\mathcal{K}_n})$ for some decreasing sequence of non-zero ideals $\mathcal{I}_n \lhd R$. Therefore, for every $ \varepsilon > 0$ there exists some $n \in \N$ and a vector $w_\varepsilon \in \mathcal{K}_n$ such that $\|\xi_{\mathrm{cong}} - w_\varepsilon\| \le \varepsilon $. The set of vectors  $\pi_\varphi(U)w_\varepsilon$ is finite. For any $u \in U$ and $i \in \Z$ we have:
    $$\sup_{u \in U} \;\abs{\varphi_{\mathrm{cong}} \, (a^i u)} = \sup_{u \in U} \abs{\left<\pi_\varphi(a^iu)\xi_{\mathrm{cong}},\xi_\mathrm{cong}\right> }\leq \sup_{u \in U} \abs{\left<\pi_\varphi(a^iu) w_\varepsilon,w_\varepsilon\right> }
    + 2\varepsilon.$$
Since the supremum on the right hand side is taken over a finite collection, and since the representation  $\pi_\varphi$ is mixing on $T$, we get
$$\limsup_{i\to\infty}\sup_{u \in U} \;\abs{\varphi_{\mathrm{cong}} \, (a^i u)} \leq 2\varepsilon.$$
As the constant $\varepsilon$ was arbitrary, the desired result follows.
\end{proof}

\begin{remark}
In the setting of Proposition     \ref{prop: mixing is uniform in U}, the restriction of the GNS representation $\pi_\varphi$ to the subgroup $T$ is a mixing representation if and only if the restriction of the trace $\varphi$ to the subgroup $T$ is a mixing trace. The non-trivial direction of this claim follows from the fact that the set $\pi_\varphi(g)\xi_\varphi$ is total in $\mathcal{H}_\varphi$, and each such vector generates an isomorphic copy of the cyclic GNS representation of $T$.
\end{remark}

\begin{remark}
Assume that $F$ is an algebraic number field and that $S$ consists only of the Archimedean valuations $\mathscr R_\infty$. In that case $R$ is a ring of algebraic integers, and the solvable group $M = TU$ studied in Proposition \ref{prop:characters of AU} is polycyclic. More generally, if $\mathbf P$ is a minimal $F$-parabolic subgroup of a semisimple algebraic $F$-group, then the solvable group $\mathbf P(R)$ is virtually polycyclic. The characters of virtually polycyclic groups were completely classified from a different perspective in the work \cite{levit2024characters}. We also mention \cite{bader2024charmenability}, which is another interesting work dealing with characters of solvable arithmetic groups.
\end{remark}

\section{A mixing phenomenon}\label{sec:6}
In this section we state and prove our key Theorem \ref{theorem:on characters of G admitting AU and AV}. We use the Standing Assumptions introduced in \S\ref{standing assumption}. 
Namely, $F$ is a global field with $\mathrm{char}(F)\neq 2$ and $S \subset \mathscr R$ is a finite set of valuations of $F$ that contains all the Archimedean ones $\mathscr R_\infty$. We crucially assume that $|S| \geq 2$, so that   the ring of $S$-integral elements $R=F(S)$   has a unit of infinite order  (see Remark \ref{rem:Dirichlet}).  

Let $\mathbf G$ be a connected, simply connected, absolutely almost simple and $F$-isotropic algebraic $F$-group. Fix a  minimal $F$-parabolic subgroup $\mathbf P$  of $\mathbf G$ with unipotent radical $\mathbf U$. Let $\mathbf A$ be a maximal $F$-split torus of $\mathbf P$, and denote by $\mathbf{Q}$ be the minimal parabolic opposite to $\mathbf{P}$ and containing $\mathbf{A}$. We further denote the unipotent radical of $\mathbf{Q}$ by $\mathbf{V}$. Fix some $F$-embedding of $\mathbf{G}$ into a matrix group $\mathbf{GL}_m$ for some $m \in \N$ and denote 
$$ \Gamma = \mathbf G(R), \quad A = \mathbf{A}(R),\quad U = \mathbf{U}(R), \quad V = \mathbf{V}(R).$$ 

Given a non-zero ideal  $\mathcal{I} \lhd R$ and an $F$-algebraic subgroup $\mathbf{H}$ of $\mathbf{G}$, we let $H(\mathcal I)$ denote the corresponding congruence subgroup (with the same convention for other letters denoting algebraic subgroups).  Let $a \in A$ be an element such that no power of $a$ is in the kernel of any root $\alpha  \in \Phi(\mathbf A,\mathbf G)$; see Remark \ref{good a exists}. We consider the groups
$$T = \left<a\right>, \quad M_U = T \ltimes U \quad \text{and} \quad M_V = T\ltimes V .$$
\begin{theorem}
\label{theorem:on characters of G admitting AU and AV}
Every weakly mixing trace $\varphi \in \Tr{\Gamma}$ satisfies 
$$ \lim_{i \to \infty} \sup_{u \in U} \; \varphi  (a^i u) = 0.$$
In particular, the restriction of the trace $\varphi$ to the subgroup $T$ is mixing.
\end{theorem}
Our key Theorem \ref{theorem:on characters of G admitting AU and AV} is analogous to the main proposition of Peterson and Thom in the case of $\mathrm{SL}_2$ \cite[Proposition~2.2]{PetersonThom}. Their line of proof is adaptable to our general setting. Indeed, if the unipotent radical $\mathbf U$ of the minimal parabolic $\mathbf P$ is abelian then \cite{PetersonThom}  works almost verbatim. The general case requires more work. We chose to give a different line of proof, relying on the character theory of solvable arithmetic groups developed in \S\ref{sec:5}. Of course, our proof still shares mutual ideas with \cite{PetersonThom}. In particular, both proofs heavily rely  on the following deep result of Raghunathan and Venkataramana.
\begin{theorem}[{\cite{raghunathan1976congruence,Venkataramana}}]
\label{theorem:Raghunathan-Venkataramana}
For any non-zero ideal $\mathcal I \lhd R$ the subgroup $\left< U(\mathcal I),  V(\mathcal I)\right>$ has finite index in $\Gamma = \mathbf G(R)$.
\end{theorem}
\begin{remark}
The assumptions of the Raghunathan--Venkataramana theorem are weaker than those stated above. Namely, $S$ does not have to be of size $\geq 2$, but it is  required that $\mathrm{rank}_F(\mathbf{G}) \ge 1$ and $\sum_{\nu\in S}\mathrm{rank}_{F_\nu}(\mathbf{G}) \ge 2$.
We emphasize that Theorem \ref{theorem:Raghunathan-Venkataramana} is the only place  where we use the assumption  $\text{char}(F) \ne 2$. The authors do not know whether Theorem \ref{theorem:Raghunathan-Venkataramana} holds in characteristic two (if it does hold  then the restriction on the characteristic  would be removed from our theorem).
\end{remark}

\begin{remark}
The particular Standing Assumption $\mathrm{rank}_F(\mathbf G) = 1$ is redundant and is never used in \S\ref{sec:6}.
\end{remark}

\subsection{The interplay of induced and congruence characters}
In Proposition  \ref{prop:characters of AU} and the discussion following it, we have decomposed the Thoma space $\Ch{M_U}$ into a disjoint union of two Borel subsets $ \Ch{M_U}_\mathrm{cong}$ and $\Ch{M_U}_\mathrm{ind}$. The GNS representation of each character $\varphi \in \Ch{M_U}_\mathrm{ind}$ is induced from the subgroup $U$, and is in particular infinite-dimensional. On the other hand, as the ring $R$ has only countably many ideals, the subset $\Ch{M_U}_\mathrm{cong}$ of the Thoma dual admits a countable filtration, according to which congruence subgroup  of $U$ belongs to the kernel. Each  character $\varphi \in \Ch{M_U}_\mathrm{cong}$ is finite-dimensional by Proposition \ref{prop:characters of AU}. 

\begin{lemma}\label{lem: decomposition lemma}
Let $\varphi \in \Tr{\Gamma}$ be a trace with GNS-triple $(\mathcal{H}_\varphi, \pi_\varphi, \xi_\varphi)$. Then $\pi_\varphi$ admits a direct sum decomposition of $M_U$-representations $$\pi_\varphi = \pi_{M_U,\mathrm{ind}} \oplus \pi_{M_U,\mathrm{cong}}$$
such that the restriction of  $\pi_{M_U,\mathrm{cong}}$ to the subgroup $U$ is a direct sum of congruence representations, and
the restriction of $\pi_{M_U,\mathrm{ind}}$ to the subgroup $T$ is mixing.
\end{lemma}

\begin{proof} 
Let $\psi \in \Tr{M_U}$ denote\footnote{We use the different symbol $\psi$ to denote the restriction of $\varphi$ in order to avoid confusion.} the restriction of the trace $\varphi$ to the subgroup $M_U$. Let $(\mathcal{H}_\psi,\pi_\psi,\xi_\psi)$ be the GNS-triple of the trace $\psi$, so that $\pi_\psi$ is a unitary representation of $M_U$.
 According to Proposition \ref{prop:disj traces and gns} there is a 
convex combination $\psi = t \psi_\mathrm{cong} + (1-t)\psi_\mathrm{ind}$ for some $t \in \left[0,1\right]$ where $\psi_\mathrm{cong} \in \Tr{M_U}_\mathrm{cong}, \psi_\mathrm{ind} \in \Tr{M_U}_\mathrm{ind}$ as well as a corresponding direct sum  of \emph{disjoint}
$M_U$-representations $\pi_\psi = \pi_{\psi,\mathrm{cong}} \oplus \pi_{\psi,\mathrm{ind}}$. The restriction of the $M_U$-representation $\pi_{\psi,\mathrm{cong}}$ to the subgroup $U$ is a direct sum of congruence representations, and the restriction of the $M_U$-representation $\pi_{\psi,\mathrm{ind}}$ to the subgroup $T$ is mixing by Proposition \ref{prop:characters of AU}.

To conclude the proof, we need to pass from the smaller GNS representation $\pi_\psi$ to the larger one $\pi_\varphi$. These two representations satisfy $\pi_\varphi \le \infty \pi_\psi$, where $ \infty \pi_\psi$ stands for the infinite countable multiple of $\pi_\psi$  \cite[Proposition 8.2.2]{Popa} \footnote{In the special setting of GNS representations, the statement $\pi_\varphi \le \infty \pi_\psi$ can be proved directly by noting that each vector in the total set $\pi_\varphi(g) \xi_\varphi, g\in \Gamma$ generates a cyclic $M_U$-representation isomorphic to the GNS-representation $(\mathcal{H}_\psi,\pi_\psi,\xi_\psi)$.}. By the above discussion $\infty \pi_\psi = \infty \pi_{\psi,\mathrm{cong}} \oplus \infty \pi_{\psi,\mathrm{ind}}  $ and these two summands are disjoint. This induces the desired direct sum decomposition of $M_U$-representations $\pi_\varphi = \pi_{M_U,\mathrm{ind}} \oplus \pi_{M_U,\mathrm{cong}}$. The stated properties of those two summands follow from the fact that $\pi_{M_U,\mathrm{ind}} \le \infty\pi_{\psi,\mathrm{ind}} $ and likewise $\pi_{M_U,\mathrm{cong}} \le \infty\pi_{\psi,\mathrm{cong}} $.
\end{proof}

Of course, Lemma \ref{lem: decomposition lemma} holds true mutatis mutandis for the subgroup $M_V$ in place of $M_U$. We are now in a position to prove  Theorem \ref{theorem:on characters of G admitting AU and AV}.

\begin{proof}[Proof of Theorem \ref{theorem:on characters of G admitting AU and AV}]

Consider a weakly mixing trace $\varphi \in \Tr{\Gamma}$ with associated GNS-triple $(\mathcal{H}_\varphi,\pi_\varphi,\xi_\varphi)$.
The desired conclusion will follow directly from Proposition \ref{prop: mixing is uniform in U} once we establish that the restriction of the representation $\pi_\varphi$ to the subgroup $T$ is mixing.

Let $\mathcal{I}_n \lhd R$ be a descending sequence of non-zero ideals, such that for every non-zero ideal $\mathcal{J} \lhd R$ we have $\mathcal{I}_n \subset \mathcal J$ for all $n$ sufficiently large. Denote 
$U_n = U(\mathcal{I}_n)$ and $V_n = V(\mathcal{I}_n)$ for all $n \in \N$.
 Let $p_n$ and $q_n$ respectively denote the orthogonal projections to the subspace of $U_n$-invariant and $V_n$-invariant vectors in $\cH_\varphi$. The two sequences $p_n$ and $q_n$ are ascending, namely $p_{n} \le p_{n+1}$ and $q_n \le q_{n+1}$ for all $n \in \N$. Denote $p = \sup_n p_n$ and $ q = \sup_n q_n$. 
Note that $p_n \in \pi_\varphi(U_n)'' \le \pi_\varphi(\Gamma)''$ and likewise $q_n \in \pi_\varphi(V_n)'' \le \pi_\varphi(\Gamma)''$ for all $n \in \N$. Hence $p,q \in \pi_\varphi(\Gamma)''$ as well, by the fact that von Neumann algebras are closed in the strong operator topology.

Consider the restriction of $\pi_\varphi$ to the subgroup $T$.  According to  Lemma \ref{lemma:unitary decomposes into mixing and non-mixing}, this restriction admits a decomposition into a direct sum of $T$-representations
$$\mathcal{H}_\varphi = \mathcal{H}_{T,\mathrm{mix}} \oplus \mathcal{H}_{T,\mathrm{nm}}$$
such that $\mathcal{H}_{T,\mathrm{mix}}$ is mixing and no $T$-subrepresentation of $\mathcal{H}_{T,\mathrm{nm}}$ is mixing.
On the one hand, by Lemma \ref{lem: decomposition lemma} there exists a canonical decomposition into $M_U$-subrepresentations $\mathcal{H}_\varphi = \mathcal{H}_{M_U,\mathrm{ind}} \oplus \mathcal{H}_{M_U,\mathrm{cong}}$ 
with $\cH_{M_U, \mathrm{ind}} \subseteq \cH_{T,\mathrm{mix}}$. Equivalently $$\mathcal{H}_{T,\mathrm{nm}} \subseteq \mathcal{H}_{M_U,\mathrm{cong}} = \mathrm{Im}(p).$$
On the other hand, by applying the same argument with respect to the subgroup $M_V$ we get a canonical decomposition into $M_V$-subrepresentations
$\mathcal{H}_\varphi = \mathcal{H}_{M_V,\mathrm{ind}} \oplus \mathcal{H}_{M_V,\mathrm{cong}}$
with $\cH_{M_V, \mathrm{ind}} \subseteq \cH_{T,\mathrm{mix}}$. Equivalently 
$$\mathcal{H}_{T,\mathrm{nm}} \subseteq \mathcal{H}_{M_V,\mathrm{cong}} = \mathrm{Im}(q).$$ 

To conclude the proof, assume towards contradiction that $\mathcal{H}_{T,\mathrm{mix}} \lneq \mathcal{H}_\varphi$. This assumption is equivalent to $\mathcal{H}_{T,\mathrm{nm}} \neq \{0\}$. By the above two inclusions, we get
$$\mathcal{H}_{M_U,\mathrm{cong}} \cap \mathcal{H}_{M_V,\mathrm{cong}}  = \mathrm{Im}(p) \cap \mathrm{Im}(q) \neq \{0\}.$$
According to Lemma \ref{lemma:very cool lemma} there is some  $n \in \N$ such that $\mathrm{Im}(p_n) \cap \mathrm{Im}(q_n) \neq \{0\}$. In other words, there is a non-zero vector $ \xi \in \mathcal{H}_\varphi$ invariant under $\pi_\varphi(U_n)$ as well as  $\pi_\varphi(V_n)$. The two subgroups $U_n$ and $V_n$ generate a finite-index subgroup of $\Gamma$ by Theorem \ref{theorem:Raghunathan-Venkataramana}. Hence  $\mathrm{span}(\pi_\varphi(\Gamma)\xi)$ is a non-zero finite-dimensional subrepresentation of $\mathcal{H}_\varphi$. This is a contradiction to the assumption that the trace $\varphi$ is  weakly mixing.
\end{proof}

\begin{remark}
\label{remark:redisovering}
As a byproduct of the above proof, we  rediscover the fact that every finite-dimensional unitary representation of the lattice $\Gamma$ factors through a finite quotient. This is known to be true for every higher-rank non-uniform lattice, as a corollary of Margulis' superrigidity theorem. To see how this fact follows from our proof, take any irreducible finite-dimensional unitary representation $\rho$ of $\Gamma$ and consider its trace $\varphi = \mathrm{tr}\circ \rho \in \Ch{\Gamma}$. The GNS representation $\pi_\varphi$ is finite-dimensional,  hence its restriction to $T$ is not mixing. The above argument produces a subrepresentation of $\pi_\varphi$ consisting of invariant vectors for some finite-index subgroup of $\Gamma$. Therefore $\pi_\varphi$ admits a non-trivial subrepresentation factoring through a finite quotient of $\Gamma$. The desired fact now follows, as  $\pi_\varphi$ is equivalent to a finite multiple of $\rho$.
\end{remark}

In the proof of Theorem \ref{theorem:on characters of G admitting AU and AV} we made an essential use of the following fact on projections in tracial von Neumann algebras. To the best of our understanding, it has been used implicitly in the proof of \cite[Proposition 2.2]{PetersonThom} by Peterson and Thom.

\begin{lemma}
\label{lemma:very cool lemma}
Let $(\mathcal M,\tau)$ be a tracial von Neumann algebra. Let $p_n, q_n \in \mathcal M$ be two sequences of projections such that $p_ n\leq p_{n+1}$ and $ q_n \leq q_{n+1}$ for all $n \in \N$. Denote $p = \sup_n p_n$ and $q = \sup_n q_n$. If $p \wedge q \neq 0$ then   $p_{n_0} \wedge q_{n_0} \neq 0$ for some $n_0 \in \N$.
\end{lemma}

\begin{proof}
To prove the lemma, it will be enough to show that if $p \wedge q \neq 0$, then there exists some $n\in\N$ such that $p_n \wedge q \neq 0$. Indeed, fixing such $n \in \N$, we can apply the argument a second time to obtain some $m \in \N$ such that $p_n \wedge q_m \neq 0$. 
The lemma would now follow by taking $n_0 = \max\{n,m\} \in \N$.

Kaplansky's formula \cite[Proposition 2.4.5]{Popa} says that  any two projections $a,b \in \mathcal M$ satisfy
$$a \vee b  - b \sim  a - a \wedge b.$$ 
The symbol  $\sim$ stands for Murray-von Neumann equivalence \cite[Definition 2.4.3]{Popa}. This equivalence implies in particular that $\tau(a \vee b  - b)= \tau(  a - a \wedge b)$.
Applying Kaplansky's formula with respect to 
the elements $a = p \wedge q$ and $b= p_n$ yields
    $$(p \wedge q) \vee p_n - p_n \sim p \wedge q - (p \wedge q) \wedge p_n =  p \wedge q - p_n \wedge q.$$
The sequence $(p \wedge q ) \vee p_n$ is an increasing sequence of projections, whose supremum is $(p \wedge q) \vee p = p$. 
    It follows that $(p\wedge q) \vee p_n - p_n$ converges in the strong operator topology to $0$.
    Consequently, by the normality of the trace $\tau$ we have
    $$ \lim_n \tau(p \wedge q )- \tau (p_n \wedge q) = \lim_n \tau(p \wedge q -p_n \wedge q)=\lim_n \tau((p \wedge q) \vee p_n - p_n ) = 0.$$
The faithfulness of the trace $\tau$ implies $\tau(p \wedge q) \neq 0$. Therefore  there exists some $n \in \N$ with $\tau(p_n \wedge q) \neq 0$, as desired.
\end{proof}

\section{Two proofs of character rigidity}\label{sec:7}
In this section we establish character rigidity and prove Theorem \ref{main_thm}. 
Let $H$ be a semisimple group as in Theorem \ref{main_thm}, and $\Lambda$ an irreducible non-uniform lattice in $H$. According to Proposition \ref{prop:notationIsGood}, there exist a global field $F$, a finite set of valuations $S$ and an $F$-group $\mathbf{G}$ that satisfy all of the Standing Assumptions \S\ref{standing assumption}, such that $\Gamma=\mathbf{G}(R)=\mathbf{G}(F(S))$ is a central extension of some lattice commensurable to $\Lambda$. Proposition \ref{prop:char_rig_extensions} and Corollary \ref{cor:char_rig_commensuability_nonICC} imply that $\Gamma$ is character rigid if and only if $\Lambda$ is. The results from \S\ref{sec:4} are  applicable, for the lattices $\Gamma$ and $\Lambda$ have at most countably many finite-dimensional unitary  by \cite[Proposition~7.1]{BBHP}. Hence, it will suffice to prove character rigidity for the arithmetic lattice $\Gamma$, which is what we proceed to do. 

We present two different proofs of character rigidity for $\Gamma$, both relying on Theorem \ref{theorem:on characters of G admitting AU and AV}. The first proof is self-contained, while the second one uses the notion of charmenability, which is discussed below.

\subsection{The first proof}\label{subsec:7.1}

We will use the generalized Bekka lemma (Lemma \ref{lemma:generalized bekka lemma}) together with our   mixing theorem (i.e. Theorem \ref{theorem:on characters of G admitting AU and AV}) in order to deduce character rigidity. A novel idea is to apply Theorem \ref{theorem:on characters of G admitting AU and AV} with respect to different tori, the particular choice of a torus depending on the group element in question.

\begin{remark}
The proof will use the properties of the action of the group $\mathbf G(F)$ satisfying $\mathrm{rank}_F(\mathbf G) = 1$ on the countable set $B = \mathbf G(F) / \mathbf P(F)$. We will informally call the set $B$ the \enquote{boundary}, even though there is no topology or measurable structure present. 
\end{remark}

\begin{remark}
\label{remark:acting non-trivially}
Any non-central element $g \in \mathbf G(F) \setminus \mathrm{Z}(\mathbf G(F))$ acts non-trivially on the boundary $B$. This follows from the fact that the group $\mathbf G$ is almost $F$-simple. 
Indeed, if an element $g$ fixes the boundary $B$ pointwise  then $ g \in \bigcap_{h \in \mathbf G(F)} h \mathbf P(F) h^{-1}$. The latter is a proper normal subgroup of $\mathbf G(F)$. By \cite[Theorem~I.1.5.6]{Margulis}  this normal subgroup is contained in the center $\mathrm{Z}(\mathbf G(F))$.
\end{remark}

\begin{remark}
Most of the results that we use follow from the Bruhat decomposition $\mathbf G(F) = \mathbf P(F) \amalg \mathbf P(F)w\mathbf P(F)$, where $w$ is a non-trivial Weyl element associated to some $F$-split torus $\mathbf A \leq \mathbf P$, see \cite[Theorem 21.15]{Borel}. 
We caution the reader that the Bruhat decomposition of an element $g \in \Gamma$ might  involve elements from $\mathbf{G}(F)$ that are not from $\Gamma = \mathbf{G}(R)$.
\end{remark}

\begin{remark}
We write $\left[g,h\right] = g^{-1}h^{-1}gh$ following the convention used in \cite{dogon2024characters}.
\end{remark}

\begin{proof}[First proof of Theorem \ref{main_thm}]

Let $\varphi \in \Ch{\Gamma}$ be an infinite-dimensional character.
Choose any element $g \in \Gamma \setminus \mathrm{Z}(\Gamma)$. In order to establish that $\Gamma$ is character rigid we  show  that $\varphi(g) = 0$.

The element $g$ acts non-trivially on the boundary $B$ (see Remark \ref{remark:acting non-trivially}).
Choose some point $\xi \in B$ in the boundary not fixed by $g$. We now get two distinct boundary points $\xi \neq g\xi$. Since $\mathrm{rank}_F(\mathbf{G})=1$,  these two points correspond to a pair of opposite minimal $F$-parabolic subgroups $\mathbf Q$ and $\mathbf Q'$ stabilizing $\xi$ and $g\xi$ respectively. The intersection $\mathbf{L} = \mathbf Q \cap \mathbf Q'$ is a mutual Levi subgroup of both parabolics. The Levi subgroup $\mathbf L$ contains a unique 
central
$F$-split torus $\mathbf A$. Recall that $\mathbf Q$ admits a Levi decomposition  $\mathbf{Q} = \mathbf{LU}$ where $\mathbf{U}$ is the unipotent radical of $\mathbf{Q}$. 

The Weyl group of the  torus $\mathbf A$ is $\mathrm{N}_{\mathbf G}( \mathbf A) / \mathrm{Z}_{\mathbf G}(\mathbf A) \cong \Z/2\Z$. Choose a Weyl element $w \in \mathrm{N}_{\mathbf G}(\mathbf A)(F)$ representing the non-trivial class. The action of $w$ on the torus $\mathbf A$ is given by $waw^{-1} = a^{-1}$. In its action on the boundary, the element $w$ flips the two fixed points $\xi$ and $g\xi$.
Since $g\xi = w\xi$ and as $\mathbf Q = \mathrm{Stab}_{\mathbf G}(\xi)$,  we conclude that $g = wb$ for some element $b \in \mathbf Q(F)$. Using the Levi decomposition, we may write $g = wb = wmu$ for some pair of elements $m \in \mathbf L(F)$ and $u \in \mathbf U(F)$.

We are now in the right setting to apply Lemma \ref{lemma:generalized bekka lemma} (i.e. the \enquote{generalized Bekka lemma}). Choose an element $a \in A = \mathbf A(R)$ that generates a Zariski-dense subgroup of $\mathbf A$. Consider the sequence of elements $x_n = a^n$. We compute the commutators of $g$ and $x_n$:
\begin{align*}
\left[g,x_n\right] &= 
\left(wmu\right)^{-1} a^{-n}  \left(wmu\right) a^n  =  \left(mu\right)^{-1}  \left(w^{-1} a^{-n} w\right) \left(mu\right) a^n= \\
&=  \left(mu\right)^{-1}  a^{n}  \left(mu\right) a^{n}  = u^{-1} a^{n} u a^{n}  = a^{2n} v_n
\end{align*}
for some element $v_n  \in U =  \mathbf U(R)$. Indeed, since $\mathbf A(F)$ normalizes $\mathbf U(F)$ the element $v_n$ certainly belongs to $\mathbf U(F)$. The fact that $v_n$  moreover lies in $\mathbf{G}(R)$ follows from the fact that $\mathbf{G}(R)$ is a group. In the above computation we used the formula for the action of the Weyl group  on the torus, as well as the fact that $\mathbf L$ centralizes $\mathbf A$.

Next, we compute the product of a commutator and an inverse of another commutator, namely
$$[g,x_n]^{-1}[g,x_m] = (a^{2n} v_n)^{-1} (a^{2m} v_m) = a^{2(m-n)}v_{n,m}$$ 
for some element  $v_{n,m} \in U$. The fact that $v_{n,m}$ belongs to $ U$ follows from an argument similar to that of $v_n$.
We apply Theorem \ref{theorem:on characters of G admitting AU and AV} with respect to the solvable subgroup $M = \left<a\right> \ltimes U$ and deduce for each $m \in\N$ that
$$\lim_n |\varphi([g,x_n]^{-1}[g,x_m]) |= \lim_n |\varphi(a^{2(m-n)}v_{n,m})| \le \lim_n \sup_{u\in U(R)} |\varphi(a^{2(m-n)}u)|= 0.$$ 
Hence $\varphi(g) = 0$ by the generalized Bekka lemma (Lemma \ref{lemma:generalized bekka lemma}). 
\end{proof}

\subsection{The second proof}

We present an alternative proof for Theorem \ref{main_thm}. It also uses Theorem \ref{theorem:on characters of G admitting AU and AV}, but  instead of the Bruhat decomposition arguments in rank one, it uses the notion of charmenability and a reduction to the theorem of Peterson and Thom.

Before we continue, we briefly recall this notion of \textit{charmenability}, that was introduced by Bader, Boutonnet, Houdayer and Peterson in \cite{BBHP}.
\begin{definition}[{\cite[Definition~1.2]{BBHP}}]
\label{def:charmenable}
    A group $\Delta$ is called \textit{charmenable} if is satisfies the following two properties:
    \begin{enumerate}
        \item Every non-empty compact convex $\Delta$-invariant subset of $\mathrm{PD}_1(\Delta)$ contains a character.
        \item Every character of $\Delta$ is either supported on the amenable radical of $\Delta$, or is von Neumann amenable --- that is, the von Neumann algebra of its unitary GNS representation is amenable.
    \end{enumerate}
\end{definition}
For the arithmetic lattice $\Gamma$ introduced in the first paragraph of \S\ref{sec:7}, the amenable radical $\text{Rad}(\Gamma)$ is just the center $\mathrm{Z}(\Gamma)$, which is  finite. Therefore, if the lattice $\Gamma$ was charmenable, then every character of $\Gamma$ would  either be induced from the center or von Neumann amenable. This turns out to indeed be the case.
\begin{theorem}[{\cite[Theorem~B]{BBH}}]
\label{thm:charmenable}
The arithmetic lattice $\Gamma$ is charmenable.
\end{theorem}
Hence, in order to prove character rigidity for $\Gamma$ it will suffice to prove that every von Neumann amenable character of $\Gamma$ is finite-dimensional. We now recall the character-rigidity result of Peterson and Thom.
\begin{theorem}[{\cite[Theorem~2.6]{PetersonThom}}]\label{petersonThom}
    The group $\mathrm{SL}_2(R)$ is character rigid.
\end{theorem}
\begin{remark}
        Peterson and Thom state their result in characteristic zero, namely for localizations of number fields. However, reading  their proof, one notes that they only use the fact that the ring has an infinite order unit,  and satisfies the \enquote{strong congruence subgroup property} as stated in \cite[Lemma~2.7]{PetersonThom}. That property  is true any characteristic other than $2$ by Theorem \ref{theorem:Raghunathan-Venkataramana} (in any case, a proof of Theorem \ref{petersonThom} in all characteristics is also covered by our first proof).
        \end{remark}

We are now ready to give the second proof. We will freely use the Standing Assumptions \S\ref{standing assumption}; some of these assumptions were reiterated in the first paragraph of the current \S\ref{sec:7}.

\begin{proof}[Second proof of Theorem \ref{main_thm}]
Assume towards contradiction that some character $\varphi \in \Ch{\Gamma}$  is von Neumann amenable but is not finite-dimensional. 

There exists some algebraic $F$-morphism  $f : \mathbf{SL}_2 \to \mathbf{G}$ with finite kernel  \cite[Theorem~I.1.6.3]{Margulis}. Denote $\mathbf{L} = f(\mathbf{SL}_2)$ so that $\mathbf{L}$ is an $F$-subgroup of $\mathbf{G}$. Consider the subgroup $L = \mathbf{L}(R) \leq \Gamma$. We note that the group $L$ is character rigid. Indeed, Theorem \ref{petersonThom} guarantees that the group $\mathrm{SL}_2(R)$ is character rigid, and by Proposition \ref{prop:char_rig_extensions} so are its central quotients.
The group $L$ is abstractly commensurable
with a central quotient of $\mathrm{SL}_2(R)$. Hence $L$ is character rigid by Proposition \ref{prop:char_rig_commensurability}.

Let $\mathbf D \le \mathbf{SL}_2$ be the diagonal subgroup, so that $\mathbf{D}$ is an $F$-split torus. Consider the $F$-split torus $\mathbf A = f(\mathbf D) \le \mathbf G$  and  denote $A = \mathbf A(R)$.
The arithmetic group $\Gamma$ satisfies the assumptions of Theorem \ref{theorem:on characters of G admitting AU and AV}. Namely, since the character $\varphi$ is not finite-dimensional, we have $\lim_n \varphi(a^n) = 0$ for any element $a \in A$ generating a Zariski-dense subgroup.  
By the assumption that  $\varphi$ is von Neumann amenable, it follows that its restriction to $L$ is also von Neumann amenable \cite[Remark 11.2.3]{Popa}. 
Character rigidity for $L$ implies that $\varphi_{|L}$ is a convex combination of finite-dimensional characters.

We deduce that the unitary GNS representation $(\pi,\mathcal{H},v)$ of the restriction $\varphi_{|L}$ is a countable direct sum of finite-dimensional unitary representations (Proposition \ref{prop:disj traces and gns}). Denote $T = \left<a\right>$ where $a \in A$ is an element as above. The cyclic $T$-subrepresentation $\rho$ generated by the unit vector $v$ is also a countable direct sum of finite-dimensional unitary representations. This yields a contradiction to the fact that the $T$-representation $\rho$ is mixing by Theorem \ref{theorem:on characters of G admitting AU and AV}. 
\end{proof}

\section{Further results and implications}\label{sec:8}
In this section we prove Theorem \ref{main_thm_general} and its corollaries mentioned in the introduction, and provide some further implications.

\subsection{Synthesis with existing results}
We begin by verifying that all the various cases mentioned in Theorem \ref{main_thm_general} are covered by an amalgamation of existing results \cite{BBH,BBHP} and our Theorem \ref{main_thm}. 

\begin{proof}[Proof of Theorem \ref{main_thm_general}]
    Let $H$ be a semisimple group with $\mathrm{rank}(H) \ge 2$ and $\Gamma \leq H$ be an irreducible lattice. 
    
    If the semisimple group $H$ has some simple factor with Kazhdan's property (T) then  $\Gamma$ is known  to be character rigid by previous work. In fact $\Gamma$ enjoys  a stronger property, namely $\Gamma$ is \textit{charfinite} in the sense of \cite[Definition~1.2]{BBHP}. This is established in \cite[Theorem~A]{BBHP} in case $\Gamma$ is of \emph{product type} (namely, the semisimple group $H$ is not almost simple), and in \cite[Theorem~A]{BBH} in case $H$ is almost simple. 
    
    If the semisimple group $H$ does not  have a factor with Kazhdan's property $T$,  we further assume that the lattice $\Gamma$ is non-uniform and that the characteristic of the underlying global field is different than $2$, in which  case the lattice $\Gamma$ is character rigid by our Theorem \ref{main_thm}.
\end{proof}

\subsection{Stabilizer rigidity for measure preserving actions}
Character rigidity is known to imply the Stuck--Zimmer rigidity property for stabilizers in probability measure preserving actions, see e.g. \cite[Theorem 3.2]{PetersonThom} or \cite[Theorem 2.11]{DM}. We are interested in groups with a possibly infinite center, a case not yet treated in the literature. Hence we include a proof of the following general principle in that setting. Note that the only additional assumption needed for a character rigid group to satisfy the Stuck--Zimmer rigidity property has to do with finite-dimensional representations.

\begin{theorem}\label{thm:stuck_zimmer_char_rig}
Let $\Lambda$ be a character rigid group having at most countably many finite-dimensional unitary representations. Then every ergodic probability measure preserving action $\Lambda \acts(X,\mu)$ is either essentially transitive or satisfies $\mu(\fix(g)) = 0$ for every element $g \in \Lambda \setminus \mathrm{Z}(\Lambda)$.
\end{theorem}

\begin{proof}
Let $(X,\mu)$ be an ergodic probability measure preserving $\Lambda$-space. Consider the associated Vershik trace $\tau_\mu$ given by 
$$\tau_\mu(g) = \mu(\fix(g)) \quad \forall g \in \Lambda.$$
The corresponding unitary GNS representation $\pi_\mu$ is the representation of $\Lambda$ acting on the Hilbert space $L^2(\cR)$ via multiplication on the left coordinate, where $\cR$ is the orbit equivalence relation of the action. Let $\nu_\mu \in \mathrm{Prob}(\mathrm{\Ch{\Lambda}})$ be the Bochner transform of the Vershik trace $\tau_\mu$.

Assume that the ergodic action $\Lambda \acts (X, \mu)$ is not essentially transitive. Hence almost every orbit of the action is infinite. According to \cite[Proposition 3.1]{PetersonThom}, this implies that the trace $\tau_\mu$ is weakly mixing. Character rigidity and Lemma \ref{lemma:applying finitely many} implies that every weakly mixing trace of $\Lambda$ is supported on the center. In particular $\tau_\mu$ is supported on the center of $\Lambda$. Namely
$\tau_\mu(g) = \mu(\fix(g)) = 0$ for all $g \in \Lambda \setminus \mathrm{Z}(\Lambda)$.
\end{proof}

We may now deduce Corollary \ref{cor:SZ intro}.

\begin{proof}[Proof of Corollary \ref{cor:SZ intro}]
Let $\Gamma$ be a non-uniform higher-rank irreducible lattice as in Theorem \ref{main_thm_general}. It is character rigid by that theorem.
It was established in Theorem \ref{thm:stuck_zimmer_char_rig} that character rigidity implies the desired conclusion regarding stabilizers of probability measure preserving actions, as long as the lattice in question has at most countably many finite-dimensional unitary representations. This fact holds true for irreducible higher-rank lattices in algebraic groups, as a consequence of Margulis' superrigidity theorem; see \cite[Proposition 7.1]{BBHP}. Interestingly, in the generality of Theorem \ref{main_thm} it can be deduced directly from the results of the present work (without invoking superrigidity) that any finite-dimensional unitary representation of the lattice $\Gamma$ factors through a finite quotient (see Remark \ref{remark:redisovering}). 
\end{proof}

\begin{remark}\label{rem:charfinite}
We conclude that any lattice $\Gamma$ as in Theorem \ref{main_thm_general} is charfinite. Indeed, the works \cite{BBHP,BBH} establish that the lattice $\Gamma$ is charmenable. The extra conditions needed for a charmenable group to be  charfinite are character rigidity, finite amenable radical and having at most countably many finite-dimensional unitary representations. The latter property is provided by \cite[Lemma 7.1]{BBHP}. 
\end{remark}

\subsection{Semisimple Lie groups with infinite center}

Lastly, we aim to prove Theorem \ref{inf_ctr} for lattices in higher-rank semisimple Lie groups with a possibly infinite center. In order to apply Theorem \ref{thm:stuck_zimmer_char_rig}, we will need to show that such lattices enjoy the property of admitting  at most countably many finite-dimensional unitary representations.

\begin{lemma} \label{finitly many reps comensurability}
Let $\Lambda$ be a countable group and $\Lambda_0$ a finite-index subgroup. Then  $\Gamma$ has at most countably many finite-dimensional unitary representations if and only if $\Sigma$ does.
\end{lemma}
\begin{proof}
We leave the straightforward verification of this fact to reader, by means of inductions and restrictions of unitary representations.
\end{proof}

\begin{lemma} \label{countably many reps central extension}
    Let $\Gamma$ be a finitely generated group and $K \leq \Gamma$ a finite central subgroup. Denote $\Delta = \Gamma / K$. If $\Delta$ has at most countably many finite-dimensional unitary representations then so does $\Gamma$.
\end{lemma}

\begin{proof}
We may assume without loss of  generality  that the group $\Gamma$ is separated by its finite-dimensional unitary representations. 
Otherwise, we may replace $\Gamma$ with its quotient  by the normal subgroup consisting  of all elements that are trivial in every finite-dimensional unitary representation. 
    
Since $\Gamma$ is finitely generated and its elements are separated by finite-dimensional unitary representations, it is residually finite.
    Therefore there exists a finite group $F$ and a homomorphism  $\alpha:\Gamma \to F$ such that $\alpha_{|K}$ is injective. Thus $\Gamma$ is isomorphic to a finite-index subgroup of the product $\Delta \times F$ under the diagonal embedding. 
    
    The product group $\Delta \times F$ has countably many finite-dimensional unitary representations, as any irreducible unitary representation of $\Delta \times F$ is a tensor product of irreducible representations of $F$ and $\Delta$. 
  This property passes to the subgroup $\Gamma$ by Lemma \ref{finitly many reps comensurability}.
\end{proof}

We  return to the setting of Theorem \ref{inf_ctr}. Let $H$ be a connected semisimple Lie group with $\mathrm{rank}_\mathbb{R}(H) \ge 2$, without compact factors and possibly with infinite center. Let  $\Lambda \leq H$ be an irreducible lattice, in the sense of \cite[Definition~IX.6.2]{Margulis}. We assume  that $H$ either has a factor with Kazhdan's property (T), or that the lattice $\Lambda$ is non-uniform.

\begin{lemma}
\label{lem:Lambda with infinite center has countably many fin dim}
The lattice $\Lambda$ as above has at most countably many finite-dimensional unitary representations.
\end{lemma}
\begin{proof}
It will be enough to show that any irreducible finite-dimensional unitary representation $\pi: \Lambda \to U(d)$ satisfies 
\sloppy
$\left[\mathrm{Z}(\Lambda): \ker(\pi) \cap \mathrm{Z}(\Lambda)\right] < \infty$.
Indeed, given any finite-index subgroup $T \le \mathrm{Z}(\Lambda)$,  the quotient $\Lambda / T$  is a finite central extension of the lattice $\Ad(\Lambda)$. The group $\Ad(\Lambda)$ is a lattice in a higher-rank semisimple algebraic $\R$-group, and as such it admits at most countably many finite-dimensional unitary representations by \cite[Proposition~7.1]{BBHP}. Lemma \ref{countably many reps central extension} implies that the same fact is true for the quotient $\Lambda /T$ .

Assume by contradiction that there is an irreducible finite-dimensional unitary representation $\pi: \Lambda \to U(d)$ satisfying $|\pi(\mathrm{Z}(\Lambda))| = \infty$.
Since $\pi$ is irreducible, Schur's lemma guarantees that $\pi(\mathrm{Z}(\Lambda))  \leq U(d)$ consists of scalar matrices.
It follows that the image of the homomorphism $\det \circ \pi: \Lambda \to S^1$ is infinite.
 This contradicts the fact that $\Lambda$ has finite abelianization \cite[\S IX.6.18(B)]{Margulis}. This concludes the proof.
\end{proof}

\begin{proof}[Proof of Theorem \ref{inf_ctr}]
The adjoint Lie group $\Ad(H)$ is the identity component of the algebraic $\mathbb{R}$-group $G = \overline{\Ad(H)}^\mathrm{Z}$ \cite[Theorem I.2.3.1]{Margulis}. The group $G$ has no compact factors, and it has a factor with Kazhdan's property (T) if and only if $H$ does.
It follows from \cite[$(a)$ and $(b)$ in Lemma IX.6.1]{Margulis} that $\Ad(\Lambda) \leq \overline{\Ad(H)}^{\mathrm{Z}}$ is an irreducible lattice.
Thus $\Ad(\Lambda)$ is character rigid by Theorem \ref{main_thm_general}. Moreover $\Ad(\Lambda)$ has at most countably many finite-dimensional unitary representations by \cite[Proposition~7.1]{BBHP}.
Proposition \ref{prop:char_rig_extensions} implies that the lattice $\Lambda$ is character rigid as well, being a central extension of $\Ad(\Lambda)$. The fact that the  lattice $\Lambda$ satisfies the conclusion of Corollary \ref{cor:SZ intro} (i.e. it enjoys the Stuck--Zimmer stabilizer rigidity property) follows from  Theorem \ref{thm:stuck_zimmer_char_rig} combined with Lemma \ref{lem:Lambda with infinite center has countably many fin dim}.
\end{proof}

\appendix

\section{Classification of lattices in characteristic zero}\label{appen}

Assume that $H$ is a semisimple group\footnote{The notion of a semisimple group in our sense is defined above on p. \pageref{def page: semisimple group} of the introduction.} in characteristic zero and that no simple factor of $H$ has Kazhdan's property (T). In particular $H$ has  no compact factors. If the group $H$ admits an irreducible  non-uniform lattice, then at least one of the factors of $H$ must be locally isomorphic to the real Lie group $\mathrm{SO}(n,1)$ for some $n \geq 2$ or $\mathrm{SU}(n,1)$ for some $n \ge 1$. This forces the algebraic $F$-group $\mathbf{G}$ introduced in our Standing Assumptions \S\ref{standing assumption} to have the same Tits index as the group $\mathrm{SO}(n,1)$ or $\mathrm{SU}(n,1)$. Hence the group $\mathbf G$ must come from a very restrictive list. 

In this appendix we briefly recall the classification of semisimple algebraic groups by means of the Tits index and the anisotropic kernel, and give the full list of all possible (arithmetic) irreducible non-uniform lattices in products of rank-one groups without Kazhdan's property (T) and in characteristic zero.

\subsection{Classification of semisimple algebraic groups}\label{subsec:appendix1} 

Algebraic groups over algebraically closed fields are classified up to central isogeny by a \textit{Dynkin diagram}. 

The classification over an arbitrary field $k$ with algebraic closure $K$ is more complicated. For this, let $\mathbf{H}$ be a $k$-algebraic group, with maximal torus $\mathbf{T}$ and maximal $k$-split torus $\mathbf{S}$. We denote by $\Delta$ a set of simple roots of $\mathbf{H}$, and by $\Delta_0$ the set of roots in $\Delta$ which vanish on the $k$-split torus $\mathbf{S}$. We denote by $\mathcal{D}$ the Dynkin diagram of $\mathbf{H}$. There is an action  of the Galois group $\mathrm{Gal}(K/k)$ on $\Delta$  called the \textit{$*$-action} \cite[2.2]{Tits}. The  $*$-action is via automorphisms of the Dynkin diagram. Further, the $*$-action preserves the subset $\Delta_0$, and two roots in $\Delta-\Delta_0$ have the same restriction to $\mathbf{S}$ if and only if they are in the same orbit of the $*$-action. The data ($\mathcal{D}$, $\Delta_0$, $\mathrm{Gal}(K/k) \acts \mathcal{D}$) is called the \textit{Tits index} of the $k$-group $\mathbf{H}$. 

The \textit{semisimple anisotropic kernel} of $\mathbf{H}$ is  the derived subgroup of the Levi subgroup of the minimal $k$-parabolic subgroup of $\mathbf{H}$. It is denoted $\mathrm{DZ}(\mathbf S)$. 

The celebrated classification theorem of Tits \cite[Theorem~2.7.1]{Tits} is the following.
\begin{theorem}[Tits]\label{thm:classification of algebraic groups}
    A semisimple group defined over $k$ is determined up to $k$-isomorphism by its $K$-isomorphism class, its Tits index and its semisimple anisotropic kernel.
\end{theorem}

The possible Tits indices were studied by Tits. For each non-exceptional Tits index, he describes a classical representative of the isogeny class  \cite[Table~II]{Tits}.

Given a Tits index of some semisimple algebraic group $\mathbf{L}$ defined over a field $k$, and a field extension $k'$  of $k$, we can deduce the possible Tits indices of $\mathbf{L}$ regarded as an $k'$-group: roughly speaking, orbits of the Galois action may split, and the set $\Delta_0$ may become smaller. In the other direction, when one knows the Tits index over $k'$ and wants to understand the possible indices over $k$ (and hence the possible $k$-forms), roughly speaking, the orbits and the set $\Delta_0$ may now become bigger.
We refer the reader to \cite{Tits} for the complete details.

\subsection{Tits indices for the relevant ${F}$-forms}

We turn to the situation which is relevant for us. Namely, let $F$ be a number field and $\mathbf{H}$ be an $F$-isotropic algebraic $F$-group. Let $S \subset \mathscr R$ be a set of valuations of $F$ that contains all the Archimedean valuations $\mathscr R_\infty$. We denote by $R=F(S)$ the ring of $S$-integral elements of $F$, and write $\Gamma=\mathbf{H}(F(R))$. We assume  $\abs{S} \geq 2$ so that $\Gamma$ is an irreducible lattice in a product group. 

We further restrict our attention to those cases where character rigidity for the lattice $\Gamma$ does not follow from existing works that rely on Kazhdan's property (T). Namely, we assume that $\prod_{v \in S} \mathbf{H}(F_v)$ has no factor with property (T). Since we focus on non-uniform lattices, our product group has some Archimedean factor, which must be locally isomorphic to either $\mathrm{SO}(n,1)$ for some $n \geq 2$ or  $\mathrm{SU}(n,1)$ for some $n \geq 1$.

We  now use the classification Theorem \ref{thm:classification of algebraic groups} and Tits' table \cite[Table~II]{Tits} in order to find the relevant $F$-groups. Namely, we will classify $F$-isotropic algebraic $F$-groups $\mathbf{H}$ such that $\mathbf{H}(F_v)$ is locally isomorphic to either $\mathrm{SO}(n,1)$ or  $\mathrm{SU}(n,1)$ for some  Archimedean valuation $v \in S$.

We consider the $\mathrm{SO}(n,1)$ case first. There are three possible Tits indices for $\mathrm{SO}(n,1)$, depending on the value of $n$ mod $4$. We represent them diagrammatically, in the manner explained above.
\begin{enumerate}
    \item If $n \equiv 1$ mod $4$, the Tits index is:
    \[
    \begin{tikzpicture}
        \path[-] (0,0) node (a1) {$\bullet$}
        (0.5,0) node (dots) {$\cdots$}
        (1,0) node (a2) {$\bullet$}
        (2,0) node (a3) {$\bullet$}
        (3,0.5) node (a4) {$\bullet$}
        (3,-0.5) node (a5) {$\bullet$};
        \draw (a2) -- (a3) -- (a4);
        \draw (a3) -- (a5);
        \draw (a1) circle[radius=2.3mm];
        circle[radius=2.3mm]
    \end{tikzpicture}
    \]
    \item If $n \equiv 3$ mod $4$, the Tits index is:
    \[
    \begin{tikzpicture}
        \path[-] (0,0) node (a1) {$\bullet$}
        (0.5,0) node (dots) {$\cdots$}
        (1,0) node (a2) {$\bullet$}
        (2,0) node (a3) {$\bullet$}
        (3,0.5) node (a4) {$\bullet$}
        (3,-0.5) node (a5) {$\bullet$};
        \draw (a2) -- (a3) -- (a4);
        \draw (a3) -- (a5);
        \draw (a1) circle[radius=2.3mm];
        \draw [<->] (3.5,0.5) to [out=-45,in=45] (3.5,-0.5);
    \end{tikzpicture}
    \]
    \item If $n$ is even, the Tits index is:
    \[
    \begin{tikzpicture}
        \path[-] (0,0) node (a1) {$\bullet$}
        (0.5,0) node (dots) {$\cdots$}
        (1,0) node (a2) {$\bullet$}
        (2,0) node (a3) {$\bullet$}
        (2.5,0) node (arrow) {$\implies$}
        (3,0) node (a4) {$\bullet$};
        \draw (a2) -- (a3);
        \draw (a1) circle[radius=2.3mm];
    \end{tikzpicture}
    \]
\end{enumerate}
In light of the assumptions on the algebraic $F$-group $\mathbf{H}$, and  following the explanation in the last paragraph of \S\ref{subsec:appendix1}, we deduce that the $F$-Tits index of $\mathbf{H}$ is the same as its $F_v$-Tits index (except maybe the Galois group acts non trivially on the diagram $D_n$ over $F$ and not over $F_v$).
By the classification of Tits \cite[Table~II]{Tits}, the group $\mathbf{H}(F)$ is up to isogeny  $\mathrm{SO}(q)$, where $q$ is some quadratic form over $F^{n+1}$ with index one. 
Tits' notations for these indices are  $^2\text{D}_{m,1}$, $^1\text{D}_{m,1}$ or $\text{B}_{m,1}$, where $m=\lfloor \frac{n+1}{2} \rfloor$. 
\begin{proposition}\label{Form1}
    Any irreducible non-uniform lattice in a semisimple group admitting a factor locally isomorphic to the real semisimple Lie group  $\mathrm{SO}(n,1)$, has a central extension which is commensurable to the arithmetic group $\mathbf{H}(R)$, where:
    \begin{enumerate}
        \item $F$ is a number field and $R=F(S)$ is a localization of its ring of integers. If the field $F$ is either $\mathbb{Q}$ or an imaginary quadratic field, then the localization is necessarily non trivial.
        \item $\mathbf{H}$ is the algebraic $F$-group $\mathrm{SO}(q)$, where $q$ is a quadratic form over $F^{n+1}$ with index $1$.
    \end{enumerate}
\end{proposition}
A word may be put here to explain the two special cases  $\mathrm{SO}(3,1)$ and $\mathrm{SO}(7,1)$, that  might look different at first sight. The $F$-forms of $D_4$ might be exotic, but the relevant forms described above have the following Tits index:
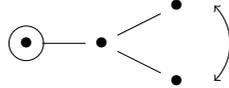
\[
\begin{tikzpicture}
    \path[-] (0,0) node (a1) {$\bullet$}
    (1,0) node (a2) {$\bullet$}
    (2,0.5) node (a3) {$\bullet$}
    (2,-0.5) node (a4) {$\bullet$};
    \draw (a1) -- (a2) -- (a3);
    \draw (a2) -- (a4);
    \draw (a1) circle[radius=2.3mm];
    \draw [<->] (2.5,0.5) to [out=-45,in=45] (2.5,-0.5);
\end{tikzpicture}
\]
Its Tits notation  is $^2D_{4,1}$ and it yields classical groups. Indeed,  the exotic ones come from triality, namely from Tits indices of the form $^3D_4$ and $^6D_4$. 

Another remark concerning the case $\mathrm{SO}(3,1)$. It is sometimes excluded from classifications of lattices in $\mathrm{SO}(n,1)$ because lattices in the $\mathrm{SO}(3,1)$ case do not  come from totally real number fields. Since we work with semisimple groups that have more than a single simple factor, and as  we did not make any claims  regarding to whether or not $F$ is totally real, our treatment covers $\mathrm{SO}(3,1)$.

Lastly, we turn to the case of $\mathrm{SU}(n,1)$. Its Tits index is given diagrammatically by:
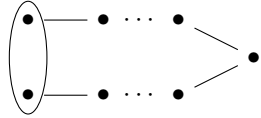
\[
\begin{tikzpicture}
    \path[-] (0,0) node (a1) {$\bullet$}
    (1,0) node (a2) {$\bullet$}
    (1.5, 0) node (dts1) {$\cdots$}
    (2, 0) node (a3) {$\bullet$}
    (3, -0.5) node (a4) {$\bullet$}
    (2, -1) node (a5) {$\bullet$}
    (1.5, -1) node (dts2) {$\cdots$}
    (1, -1) node (a6) {$\bullet$}
    (0, -1) node (a7) {$\bullet$}
    (0, -0.5) node (c) {};
    \draw (a1) -- (a2);
    \draw (a3) -- (a4) -- (a5);
    \draw (a6) -- (a7);
    \draw (c) ellipse[x radius=2.5mm,y radius=7.5mm];
\end{tikzpicture}
\]
Let $\mathbf{H}$ be an $F$-isotropic algebraic $F$-group, such that the $F_v$-Tits index of $\mathbf{H}$ is as above for some Archimedean valuation $v \in S$. As explained above, the $F$-Tits index of $\mathbf{H}$ must be the same as its $F_v$-Tits index. In Tits' notation this index is $^2\text{A}_{n,1}$ . Therefore,  by the description of Tits, we obtain:
\begin{proposition}\label{Form2}
   Any irreducible non-uniform lattice in a semisimple group admitting a factor locally isomorphic to the  real semisimple Lie group   $\mathrm{SU}(n,1)$, has a central extension commensurable to the arithmetic group $\mathbf{H}(R)$, where:
    \begin{enumerate}
        \item $F$ is a number field and $R=F(S)$ is a localization of its ring of integers. If the field $F$ is either $\mathbb{Q}$ or an imaginary quadratic field, then the localization is necessarily non trivial.
        \item $E$ is a quadratic field extension of $F$.
        \item $\mathbf{H}$ the algebraic $F$-group $\mathrm{SU}(q)$, where $q$ is an Hermitian form over $E^{n+1}$ with index $1$.
    \end{enumerate}
\end{proposition}
\bibliographystyle{amsplain}
\bibliography{bibli}

\vspace{0.5cm}

\noindent{\textsc{Department of Mathematics, Weizmann Institute of Science, Israel}}

\noindent{\textit{Email address:} \texttt{alon.dogon@mail.huji.ac.il}} \\
\noindent{\textit{Email address:} \texttt{michael.glasner@weizmann.ac.il}} \\
\noindent{\textit{Email address:} \texttt{yuval.gorfine@gmail.com}} \\
\noindent{\textit{Email address:} \texttt{liamhanany@gmail.com}} \\

\noindent{\textsc{Pure Mathematics Department, Tel Aviv University, Israel}}

\noindent{\textit{Email address:} \texttt{arielevit@tauex.tau.ac.il}} \\

\end{document}